\documentclass[a4paper,11pt, numbers]{elsarticle}
 
\usepackage{mathrsfs}

\newtheorem{theorem}{Theorem}

\newtheorem{corollary}{Corollary}
\newtheorem{proposition}{Proposition}
\newdefinition{assumption}{Assumption}
\newdefinition{definition}{Definition}
\newproof{proof}{Proof}
\newproof{pot}{Proof of Theorem \ref{thm2}}

\usepackage[colorlinks=true,breaklinks=true,bookmarks=true,urlcolor=blue,
     citecolor=blue,linkcolor=blue,bookmarksopen=false,draft=false]{hyperref}
     
\usepackage{multirow}
\usepackage[graphicx]{realboxes}
\usepackage{graphicx}
\usepackage{adjustbox}
\usepackage{bm}
\usepackage{dsfont}
\usepackage{comment}
\usepackage{longtable}
\usepackage{multicol}
\usepackage{threeparttable}
\usepackage{tabularx}
\usepackage{algpseudocode}
\usepackage{algorithm}
\usepackage{bbm}
\usepackage{amsmath}
\usepackage{amssymb}
\usepackage[numbers]{natbib}

\usepackage[margin=1in,top=1.5cm]{geometry}

\usepackage{endnotes}
\let\footnote=\endnote

\begin{document}


\begin{frontmatter}



\title{
A Framework for Stochastic Fairness in Dominant Resource Allocation with Cloud Computing Applications
}

\author[inst1]{Jiaqi Lei}

\author[inst1]{Akhil Singla}

\affiliation[inst1]{organization={Department of Industrial Engineering and Management Sciences, Northwestern University},
            city={Evanston},
            postcode={60208}, 
            state={Illinois},
            country={USA}}

\author[inst1,inst2]{Sanjay Mehrotra}

\affiliation[inst2]{organization={Corresponding Author},
            email={email: mehrotra@northwestern.edu}}
\begin{abstract}
Allocation of limited resources under uncertain requirements often necessitates fairness considerations in computer systems. This paper introduces a distributionally robust (DR) stochastic fairness framework for multi-resource allocation, leveraging rough estimates of the mean and variance of resource requirement distributions. The framework employs a sampled approximation DR (SA-DR) model to develop the concept of stochastic fairness, satisfying key properties such as stochastic Pareto efficiency, stochastic sharing incentive, and stochastic envy-freeness under suitable conditions. We show the convergence of the SA-DR model to the DR model and propose a finitely convergent algorithm to solve the SA-DR model. We empirically evaluate the performance of our moment-based SA-DR model, which uses only rough estimates of the mean and variance of the resource requirement distribution, against alternative resource allocation models under varying levels of information availability. We demonstrate that our mean- and variance-based SA-DR model can achieve performance similar to the model with full information on the resource requirement distribution. Convergence of the sampled approximation model and comparisons across models are illustrated using data from cloud computing applications.
\end{abstract}

\end{frontmatter}




%


\section{Introduction}
Fairness is often a critical consideration when allocating a limited set of resources in computer systems \cite{Angel14,Zahedi14,Tang20,Yan21}.
A central challenge arises when users' resource requirements are \emph{unknown at decision time} and only coarsely specified \cite{zheng2024efficient, ma2024optimization}. In cloud computing, the decision maker provides users with access to shared environments featuring virtual machines and resources such as memory, CPUs, and bandwidth \cite{Angel14}. The objective in such systems is to allocate resources efficiently to minimize job completion delays while maintaining fairness. However, users' resource requirements are often uncertain at the time of job submission. For instance, a computational algorithm may have unknown memory, CPU, and bandwidth requirements at the time of job submission, and users can typically only provide rough estimates of their required resources \cite{Cortez17, Yao17}. This paper introduces the concept of stochastic fairness in resource allocation for such contexts, explores its properties, and proposes an algorithm to solve the associated stochastic fairness model. We derive insights into resource allocation decisions by comparing different models that vary in the level of information available about resource requirement uncertainty.

Several concepts of fairness have been proposed for allocating a deterministic amount of resources. 
These include Nash Product Fairness \cite{Varian74}, Asset Fairness \cite{Moulin04}, Bottleneck Fairness \cite{Dolev12}, MaxMin Fairness \cite{Bertsekas92},  Proportional Fairness \cite{Kelly98}, $(\alpha,p)$-Fairness \cite{Mo00}, and Dominant Resource Fairness (DRF) \cite{Ghodsi11} and its extension \cite{Joe13}. 

Nash Product Fairness \cite{Varian74} is also called  Competitive Equilibrium from Equal Income (CEEI), and it is obtained by maximizing the geometric mean $\prod_ix_i$, where $x_i$ is the resource allocated to user $i$. Asset Fairness aims at equalizing the resources allocated to each user \cite{Moulin04}. Bottleneck fairness provides a resource allocation in which every user either gets all the required resources or at least a bottleneck resource \cite{Dolev12}. 
MaxMin Fairness provides an allocation $x$, where no individual allocation $x_i$ can be increased without increasing some $x_j$ that is less than or equal to $x_i$ \cite{Bertsekas92}.  
An allocation $x^*$ is called proportionally fair if 
$\sum_i \frac{x_i-x_i^*}{x^*_i}\leq 0$ for any other feasible $x$ \cite{Kelly98}.
A feasible allocation $x^*$ is ($\alpha,p$)-fair if $\sum_i p_i \frac{x_i-x_i^*}{\left(x_i^*\right)^ \alpha} \leq 0$, where $p=\{p_1,p_2,...,p_N\}$ and $\alpha$ are positive 
\cite{Mo00}. 
DRF is a generalization of MaxMin-fairness to multiple resource types. The resource allocated to a user is determined by the dominant share, defined as the maximum share of any resource allocated to a user \cite{Ghodsi11}. Based on DRF, \cite{Joe13} develops unified fairness-efficiency tradeoff functions called Fairness on Dominant Shares (FDS) and Generalized Fairness on Jobs (GFJ). The literature on cloud computing with fairness in multi-resource allocation has progressed from pricing-oriented $(\alpha,p)$-fair models \citep{XuH13} to DRF-based mechanisms tailored for practical constraints-strategy-proofness \citep{Friedman14}, dynamic arrivals \citep{Kash14}, and heterogeneous servers \citep{WangLiLiang14}. In addition, \citep{XuX14} introduces fairness-utilization tradeoffs and \citep{Zhao18,Khamse18} adapt bottleneck fairness and per-server dominant-share fairness to manage multi-resource bottlenecks at scale.

Stochastic multi-resource allocation problems assume resource requirements are random. \cite{Chen12} models the stochastic allocation problem in a network as a stochastic multi-dimensional knapsack problem. \cite{Wiesemann12} considers a resource allocation model for project scheduling with a chance constraint decision framework. \cite{Funaro19} proposes a stochastic allocation mechanism for cloud computing, providing users with a selection of stochastic allocation classes to bridge the gap between reserved resources and dynamic demands in the form of shares. However, unlike our paper, the above studies only consider maximizing the profit of admitted tasks or users without considering fairness. 

\cite{Yao17} addresses the robust multi-resource allocation problem in a cloud computing context; however, their analysis is limited to worst-case scenarios, focusing on scenario demand uncertainty, box demand uncertainty, and ellipsoidal demand uncertainty. \cite{Jing19} proposes a resource allocation algorithm that models the long-term max-min fairness problem for users with deterministic task requirements while assuming full information on the distribution of task arrival times. In contrast, our study derives stochastic fairness properties within a distributionally robust framework for uncertain resource requirements. We further investigate the impact of key model parameters on resource allocation decisions, considering different levels of information about the resource requirement distribution.

\subsection{Contributions}

Our paper focuses on the stochastic multi-resource allocation problem with the objective of maximizing fairness and efficiency when users' requirements are random and only partially known. We generalize the fairness--efficiency framework from the Fairness on Dominant Shares (FDS) model of \cite{Joe13} to a \emph{distributionally robust (DR)} setting where only bounds on the \emph{means and variances} of resource requirements are available. We incorporate chance constraints to control capacity violations at a prescribed confidence level. We propose a sampled approximation of the DR model (SA-DR) as our stochastic fairness model, tailored to the moment-based ambiguity set based on the available partial information. We establish conditions under which the optimal allocations from the SA-DR model satisfy \emph{stochastic Pareto-efficiency}, \emph{stochastic sharing incentive}, and \emph{stochastic envy-freeness}. We further show that SA-DR converges to the DR model under suitable assumptions and provide a \emph{cutting-surface algorithm} with finite termination that yields an $\epsilon$-optimal solution.

To generate insights on resource allocation, stochastic fairness and efficiency, we consider different models that either ignore randomness in resource requirements (i.e., deterministic expected-value models) or require full information on the probability distribution of the required resources while varying different model parameters. When full information on the resource requirement distribution is available, we show that a resource requirement distribution---such as a triangular distribution with lower variance compared to a uniform demand distribution---yields better fairness and efficiency in resource allocations. When only mean and variance, i.e., partial information about the resource requirements, is available, we employ the SA-DR model to obtain resource allocations. 

To identify the value of information in resource allocation decisions, we compare the performance of our proposed partial-information SA-DR model with that of the alternative models. Using the CloudSim dataset \cite{Yao17},
we observe that our mean- and variance-based SA-DR model consistently performs similarly to the full-information model in terms of efficiency and fairness from the optimal allocations. Moreover, the mean- and variance-based SA-DR model demonstrated similar leftovers of each resource under the optimal allocations to the full-information model-based allocations. Thus, a decision maker can attain comparable fairness and efficiency by leveraging \emph{mean–variance information} alone, without requiring full knowledge of the resource requirement distribution.


This paper is organized as follows. In Section~\ref{sec:model}, we propose our DR model and SA-DR model. 
In Section~\ref{subsec:properties}, we investigate three key stochastic fairness properties: Stochastic Pareto-efficiency, Stochastic sharing incentive, and Stochastic envy-freeness. 
A cutting surface algorithm is provided in Section~\ref{sec:algorithm} to solve the proposed SA-DR model.
In Section~\ref{sec:Computational_Results}, we propose alternative resource allocation models and perform a comparative study to analyze the impact of various model parameters and information on fairness and efficiency. Section~\ref{sec:conclusion} concludes the paper.

\section{Proposed Model}\label{sec:model}


We first describe the concept of fairness on dominant share (FDS) introduced by \cite{Joe13}. Suppose there are $d$ users and $m$ resource types. Let $r_{ij}$ denote the amount of resource $i$ required for a job by user $j$, and $c_i$ denote the total capacity of resource $i$. 
Let $x_{j}$ be the number of jobs allocated to user $j$. The capacity constraint is given by \vspace{-2mm}
\begin{equation}\label{capacity_constr}
    \sum_{j=1}^{d} r_{i j} x_{j} \leq c_{i},\quad \forall i=1,..,m.
\end{equation}
\noindent Let $\eta_{ij}=\frac{r_{ij}}{c_i}$ be the share of resource $i$ required  by user $j$ to process a job.  Next, define $\mu_{j}=\max _{i}\left\{\eta_{ij}\right\}$ as the maximum share of any resource required by user $j$ to process a job, where $\mu_{j} x_{j}$ denote the dominant share of user $j$. Finally, let $\beta\in \mathbb{R}$ be the fairness coefficient and $\lambda\in \mathbb{R}$ be the efficiency coefficient. The function used in deterministic dominant share fairness is as follows: \vspace{-1mm}
\begin{align}
F(x)=\text{sign}(1-\beta) \left(\sum_{i=1}^{d}\left(\frac{\mu_{i} x_{i}}{\sum_{k=1}^{d} \mu_{k} x_{k}}\right)^{1-\beta}\right)^{\frac{1}{\beta}}\left(\sum_{i=1}^{d} \mu_{i} x_{i}\right)^{\lambda}.\label{fairness_deterministic}
\vspace{-2mm}
\end{align}
\noindent The FDS model maximizes \eqref{fairness_deterministic} subject to the capacity constraints \eqref{capacity_constr}: \vspace{-2mm}
\begin{equation*}\label{FDS_model}
\begin{aligned}
\max_{x} \quad &\text{sign}(1-\beta) \left(\sum_{i=1}^{d}\left(\frac{\mu_{i} x_{i}}{\sum_{k=1}^{d} \mu_{k} x_{k}}\right)^{1-\beta}\right)^{\frac{1}{\beta}}\left(\sum_{i=1}^{d} \mu_{i} x_{i}\right)^{\lambda}\\[-0.5em]
\text{s.t.}\quad &\sum_{j=1}^{d} r_{i j} x_{j} \leq c_{i}\quad \forall i=1,..,m. 
\end{aligned}
\end{equation*} 
The above model assumes that all jobs are infinitely divisible, which is typical of
many multi-resource settings \cite{Yang03}.  
Special cases defined by specific values of $\lambda$ and $\beta$ are of interest: 

(1) For $\beta=1$, fairness is ignored, and we have
$
F_{\beta=1}(x)=\left(\sum_{i=1}^{d} \mu_{i} x_{i}\right)^{\lambda}$
(See \cite{Joe13}).

(2) For $\lambda=\frac{1-\beta}{\beta}$ and $\beta\neq 1$, 
the fairness function can be simplified as 
\vspace{-2mm}
\begin{equation}\label{alpha_fairness}
 F_{\lambda=\frac{1-\beta}{\beta}}(x)=\sum_{j=1}^d \frac{\left(\mu_j x_j\right)^{1-\beta}}{1-\beta},\quad \text{where}\quad \beta\neq 1.   
\end{equation}
According to \cite{Mo00}, define the ($\alpha,p$)-fairness function as $\sum_ip_if_\alpha(x_i)$, where $f_\alpha(x):= \begin{cases}\log x, & \text { if } \alpha=1; \\ 
   \frac{x^{1-\alpha}}{1-\alpha}, & \text { otherwise.}\end{cases}$ 
Let $\alpha:=\beta$ and $p=(1,..,1)$. Then optimizing $F_{\lambda=\frac{1-\beta}{\beta}}(x)$ is equivalent to optimizing the ($\alpha,p$)-fairness function.

(3) If $\beta \rightarrow\infty$ and $\lambda=\frac{1-\beta}{\beta}$, 
$F_{\beta \rightarrow\infty}(x)=
\min \left\{\mu_1 x_1, \mu_2 x_2, \ldots, \mu_d x_d\right\}$
\cite{Ghodsi11,Kalai75}. Hence, $F_{\beta \rightarrow\infty}(x)$ approaches MaxMin fairness on dominant shares.

\subsection{Proposed Stochastic Fairness Model}\label{subsec:stoch_model}
Let $\tilde{r}_{ij}$ be the random amount of resource $i$ requested by user $j$ and let $\tilde{\mu}_j:=\max_i\{\tilde{\eta}_{ij}\}$, where $\tilde{\eta}_{ij}=\frac{\tilde{r}_{ij}}{c_i}$, be the corresponding random maximum share of a resource required by user $j$ to process a job. Let \vspace{-2mm}
\begin{equation}\label{DRO_obj}
    F(x,\xi)=\text{sign}(1-\beta) \left(\sum_{j=1}^{d}\left(\frac{\tilde{\mu}_j x_{j}}{\sum_{k=1}^{d} \tilde{\mu}_k x_{k}}\right)^{1-\beta}\right)^{\frac{1}{\beta}}\left(\sum_{j=1}^{d} \tilde{\mu}_j x_{j}\right)^{\lambda},
\end{equation}
where  $\xi$ is the random vector $(\tilde{r}_{ij})$ with probability distribution $\mathcal{P}$ supported on a set $\Xi\subseteq \mathbb{R}^{m\times d}$ and $\mathbb{E}\left[ F(x, \xi) \right] 
$ is the
corresponding expected value. We assume $\mathcal{X}$ is compact and 
for every $x\in\mathcal{X}$, $F(x, \cdot)$ is bounded; and therefore,
$\mathbb{E}\left[ F(x, \xi) \right]<\infty$. Let $y_j(\xi)=\tilde{\mu}_j x_j$, $S(\xi)=\sum_k y_k(\xi)$, and $\tilde t_j(x,\xi)=y_j(\xi)/S(\xi)$. We now describe the properties of the combined fairness and efficiency function for different fairness parameters.

\begingroup
\renewcommand{\arraystretch}{0.1}
\begin{table}[h]
\centering
\scriptsize
\caption{Specializations of the expected stochastic fairness and efficiency $\mathbb{E}[F(x,\xi)]$ from \eqref{DRO_obj} across $\beta$.}
\label{tab:stoch_fairness_cases}
\begin{tabular}{|c|c|c|}
\hline
$\beta$ & $\mathbb{E}[F(x,\xi)]$ & Name \\
\hline
$\beta=1$ &
$\displaystyle \mathbb{E}\big[S(\xi)^{\lambda}\big]$ &
Fairness ignored \\[0.4em]

$\beta\to 0$ &
$\displaystyle \mathbb{E}\Big[\exp\big(H(\tilde{\boldsymbol t})\big)\,S(\xi)^{\lambda}\Big]$ &
Entropy  \\[0.4em]

$\beta\in(0,1)$ &
$\displaystyle \mathbb{E}\Big(\big[(1-\beta)U_{\beta}(\tilde{\boldsymbol t})\big]^{\!1/\beta}\,S(\xi)^{\lambda}\Big)$ &
$\alpha$-fair ($\alpha=\beta$)  \\[0.4em]

$\beta\in(1,\infty)$ &
$\displaystyle -\mathbb{E}\Big(\big[(1-\beta)U_{\beta}(\tilde{\boldsymbol t})\big]^{\!1/\beta}\,S(\xi)^{\lambda}\Big)$ &
$\alpha$-fair ($\alpha=\beta$) \\[0.4em]

$\beta=-1$ &
$\displaystyle$ $\mathbb{E}\!\left[\frac{1}{\sum_j (\tilde{\mu}_j x_j)^2}S(\xi)^{\lambda}\right]$ & Jain's index \\[0.4em]

$\beta\in(-\infty,-1)$ &
$\displaystyle \mathbb{E}\Bigg(\Big[\sum_j \tilde t_j^{\,1-\beta}\Big]^{\!1/\beta} S(\xi)^{\lambda}\Bigg)$ &
\shortstack[c]{Generalized\\Jain-type} \\[0.4em]

$\beta\to\infty$ &
$\displaystyle -\,\mathbb{E}\!\left[S(\xi)^{\lambda+1}\,\max_{j}\frac{1}{y_j(\xi)}\right]$ &
Max ratio \\[0.4em]

$\beta\to-\infty$ &
$\displaystyle \mathbb{E}\!\left[S(\xi)^{\lambda+1}\,\min_{j}\frac{1}{y_j(\xi)}\right]$ &
Min ratio \\
\hline
\end{tabular}
\vspace{0.25em}
\begin{minipage}{0.94\linewidth}\footnotesize
\emph{Notes.} (i) $U_{\beta}(p)=\sum_j t_j^{\,1-\beta}/(1-\beta)$ is the standard $\alpha$-fair utility with $\alpha=\beta$ applied to the random share vector $\tilde{\boldsymbol t}(x,\xi)$; the identity $\sum_j \tilde t_j^{\,1-\beta}=(1-\beta)U_{\beta}(\tilde{\boldsymbol t})$ is used to factor the fairness component as in \cite{Lan2010}. (ii) Limits $\beta\to 0,\pm\infty$ follow from the known pointwise limits of $F(x,.)$ and dominated convergence under our boundedness assumptions on $\xi$ and compact $\mathcal{X}$. (iii) For $\beta=-1$, $f_{-1}(\tilde{\boldsymbol y})=\big(\sum_j \tilde t_j^2\big)^{-1}=(S^2/\sum_j y_j^2)$, so $F=f_{-1}\,S^{\lambda}=S^{\lambda+2}/\sum_j y_j^2$.
\end{minipage}
\vspace{-4mm}
\end{table}
\endgroup

Table~\ref{tab:stoch_fairness_cases} describes the stochastic fairness and efficiency function under different levels of fairness parameter $\beta$. We note that the expected fairness part under different $\beta$ can be obtained by removing the $S(\xi)^{\lambda}$ term from the simplified combined objective function. Similar to a deterministic setting, for $\beta=1$, fairness is ignored, and for $\beta \to \infty$, the function converges to a stochastic equivalent of the max-ratio function \cite{Lan2010}. We now utilize the stochastic fairness and efficiency function for the optimal resource allocations.

Our proposed stochastic fairness DR model is as follows:
\vspace{-2mm}
\begin{equation}\label{dro_model}
\begin{aligned}
\max_{x\in\mathcal{X}} \min_{\mathbb{P}\in \mathcal{P}}\quad &\mathbb{E}_\mathbb{P}[F(x,\xi)]\\[-0.5em]
\text{s.t.}\quad &\text{Prob}_\mathbb{P}\left(\sum_{j=1}^{d} \tilde{r}_{ij} x_{j} \leq c_{i}\right)\geq \theta\quad \forall i=1,..,m,\quad 
\\
\end{aligned}
\end{equation}
where $\theta$ is a pre-specified chance probability that the allocated resources do not exceed capacity \cite{Chen12}.
The chance constraints in \eqref{dro_model} account for uncertainty in resource requirements and ensure that the total allocated resources remain within capacity at the desired confidence level. The expected value of the fairness objective is maximized under the worst-case setting of an unknown probability distribution, for which only partial information is available. The set $\mathcal{P}$ represents the ambiguity set, containing all probability distributions that satisfy the given partial information. The same (unknown) probability distribution affects both the resource allocations (through the constraints) and the evaluation of the fairness objective. 

\noindent \textbf{\textit{Practically-Relevant Ambiguity Set.}} In a cloud computing setting, users can provide rough statistical estimates of the required resources, such as a finite range and bounds on the mean and variance of the desired resources \cite{schad2010runtime}. The information on the finite range of required resources suggests that $\Xi$ is bounded and that $|\xi|\leq M_\Xi$ for any $\xi\in \Xi$. Note that the diameter of $\Xi$, given by $\max_{\xi^i,\xi^j\in\Xi}|\xi^i-\xi^j|$, is at most $2M_\Xi$. The information on the upper and lower bounds of the mean, $\bar{u}_{ij}$ and $\underline{u}_{ij}$, for resource $i$ required by user $j$, and the upper bound of the variance, $\bar{\sigma}_{ij}^2$, for resource $i$ required by user $j$, is used to define the ambiguity set, specifically as 
\begin{align}
\label{ambiguity_dro}
 \mathcal{P}=\{\mathbb{P}\mid \underline{u}_{ij}\leq \mathbb{E}_\mathbb{P}[\tilde{r}_{ij}]\leq  \bar{u}_{ij}, \text{var}_{\mathbb{P}}(\tilde{r}_{ij})\leq \bar{\sigma}_{ij}^2,\forall i,j\}. 
 \vspace{-8mm}
\end{align}

\noindent \textbf{\textit{Impact of Fairness Parameter $\beta$.}} We describe the properties of the combined efficiency and fairness function ($\mathbb{E}[ F(x, \xi)]$, defined as in \eqref{DRO_obj}) and its fairness component ($\mathbb{E}[ F_{\beta}(x, \xi)]$), with respect to the fairness parameter $\beta$, where \vspace{-2mm}
\begin{equation}\label{DRO_obj_fair}
    F_{\beta}(x,\xi):=\text{sign}(1-\beta) \left(\sum_{j=1}^{d}\left(\frac{\tilde{\mu}_j x_{j}}{\sum_{k=1}^{d} \tilde{\mu}_k x_{k}}\right)^{1-\beta}\right)^{\frac{1}{\beta}}.
\end{equation}

We first quantify how the fairness parameter $\beta$ changes the combined fairness-efficiency function and its fairness component (extending Theorem $5$ of \cite{Lan2010} in a stochastic setting) when the allocation $x$ is held fixed.

\begin{proposition}[Monotonicity in \(\beta\) for fixed \(x\)]
\label{prop:beta-monotone}
Let \(x\in\mathcal{X}\) be fixed and assume \(S(\xi)=\sum_{j=1}^{d}\tilde{\mu}_j x_j>0\). For \(\beta>1\), \vspace{-2mm}
\[
F_{\beta}(x,\xi)\;=\;\operatorname{sign}(1-\beta)\Bigg(\sum_{j=1}^{d}\tilde t_j(x,\xi)^{\,1-\beta}\Bigg)^{\!\!1/\beta}\quad\text{and}\quad
F(x,\xi)\;=\;F_{\beta}(x,\xi)\,S(\xi)^{\lambda}
\]
are \emph{nonincreasing} in \(\beta\) for every realization \(\xi\). Consequently, the expected values \(\beta\mapsto \mathbb{E}[F_{\beta}(x,\xi)]\) and \(\beta\mapsto \mathbb{E}[F(x,\xi)]\) are nonincreasing on \((1,\infty)\).
\end{proposition}

\noindent For \(\beta>1\), the fairness sign is negative, so increasing \(\beta\) steepens the penalty on dispersion in the dominant-share vector \(\tilde{\boldsymbol t}(x,\xi)\). Thus, for fixed \(x\), both the fairness component and the combined fairness-efficiency objective weakly decrease with \(\beta\). \\
\noindent \textit{\textbf{Monotonicity under optimal allocations in \(\bm{\beta}\).}}
For each \(\beta>1\), consider the distributionally robust value functions \vspace{-3mm}
\[
v_{\mathrm{fair}}(\beta)\;:=\;\max_{x\in\mathcal{X}}
\ \min_{\mathbb{P}\in\mathcal{P}(x)}\ \mathbb{E}_{\mathbb{P}}\!\left[F_{\beta}(x,\xi)\right], 
\quad v(\beta)\;:=\;\max_{x\in\mathcal{X}}
\ \min_{\mathbb{P}\in\mathcal{P}(x)}\ \mathbb{E}_{\mathbb{P}}\!\left[F(x,\xi)\right],
\]
where \(F_\beta\) and \(F\) are given in \eqref{DRO_obj_fair} and \eqref{DRO_obj}, respectively, and \vspace{-2mm}
\[
\mathcal{P}(x)\; := \;\left\{\mathbb{P}\in\mathcal{P}\;\Big|\;
\mathrm{Prob}_{\mathbb{P}}\!\left(\sum_{j=1}^{d}\tilde r_{ij}x_j\le c_i\right)\ge \theta\ \ \forall i=1,..,m\right\}
\]
is the nonempty set of distributions in the ambiguity set for which the chance constraints are satisfied by \(x\).
Throughout, assume that \(F(x,\cdot)\) is bounded for each \(x\in\mathcal{X}\) and that \(S(\xi)=\sum_j \tilde\mu_j x_j>0\) under any \(\mathbb{P}\in\mathcal{P}(x)\).

\begin{proposition}[Value-function monotonicity in \(\beta\) under optimal allocations]\label{prop:value-monotonicity-beta}
On the domain \(\beta>1\), both \(v_{\mathrm{fair}}(\beta)\) and \(v(\beta)\) are (weakly) \emph{decreasing} functions of \(\beta\). That is, for any \(1<\beta_1<\beta_2\),
\[
v_{\mathrm{fair}}(\beta_2)\ \le\ v_{\mathrm{fair}}(\beta_1)
\qquad\text{and}\qquad
v(\beta_2)\ \le\ v(\beta_1).
\]
\end{proposition}

Proposition~\ref{prop:value-monotonicity-beta} shows that the optimal combined efficiency and fairness, and its fairness component (as in \eqref{dro_model}), are nonincreasing in the fairness parameter $\beta>1$. \\
\noindent \textit{\textbf{Eventual approach to the \(\beta\rightarrow\infty\) benchmark when \(\bm{\lambda=}(\bm{1-}\bm{\beta)/}\bm{\beta}\).}}
Under the setting \(\lambda=(1-\beta)/\beta\), the combined function simplifies to (let $y_j(\xi):=\tilde{\mu}_j x_j$) \vspace{-1mm}
\[
\begin{aligned}
F(x,\xi)&=-\Big(\sum_{j=1}^{d} y_j(\xi)^{\,1-\beta}\Big)^{\!1/\beta},  
\end{aligned}
\]

Proposition~\ref{prop:value-monotonicity-beta} established that both \(v_{\mathrm{fair}}(\beta)\) and \(v(\beta)\) are nonincreasing in \(\beta\) for \(\beta>1\) when \(\lambda\) is \emph{fixed}. When \(\lambda=(1-\beta)/\beta\), global monotonicity of \(v(\beta)\) can fail, yet-importantly-both the optimal fairness component and the optimal combined function \emph{approach their \(\beta\rightarrow\infty\) counterparts monotonically once \(\beta\) is large enough}. The fairness component shows this monotonic convergence $\forall \beta > 1$.

\begin{proposition}[Eventual monotone convergence]
\label{prop:eventual-monotone-infty}
Suppose \(\Xi\) is bounded, \(\mathcal{X}\) is compact, \(F(x,\cdot)\) is bounded for each fixed \(x\), and $\lambda=(1-\beta)/\beta$,

\smallskip
\noindent \emph{(a) Fairness component.} For \(\beta>1\),
\(v_{\mathrm{fair}}(\beta)\) is nonincreasing and converges \emph{monotonically} as \(\beta\rightarrow\infty\):
\[
v_{\mathrm{fair}}(\beta) \rightarrow \displaystyle -\,\mathbb{E}\!\left[\max_{j}\frac{1}{y_j(\xi)}\right].
\]

\noindent\emph{(b) Combined objective with \(\lambda=(1-\beta)/\beta\).} 
Then 
\(v(\beta)\) converges \emph{monotonically} as \(\beta\rightarrow\infty\):
\[
v(\beta) \rightarrow \displaystyle -\,\mathbb{E}\!\left[S(\xi)^{\lambda+1}\,\max_{j}\frac{1}{y_j(\xi)}\right].
\]
\end{proposition}

Part~(a) of Proposition~\ref{prop:eventual-monotone-infty} extends Proposition~\ref{prop:value-monotonicity-beta} for the fairness component by identifying a \emph{monotone limit} as \(\beta\rightarrow\infty\). Part~(b) of the proposition shows that, although global monotonicity in \(\beta\) can fail for the combined objective when \(\lambda=(1-\beta)/\beta\), the combined optimal efficiency and fairness function becomes monotone beyond a certain threshold \(\beta_0\) and approaches the max-ratio function as \(\beta\rightarrow\infty\) (Table~\ref{tab:stoch_fairness_cases}).

\subsection{Sampled Approximation DR (SA-DR) Model}\label{subsec:sample_approximation}
We now propose the sampled approximation DR model \eqref{dro_model_sampled}, which approximates the stochastic fairness DR model \eqref{dro_model}.  
Let $\xi^\omega$ be a uniformly distributed scenarios of $|\Omega|$ realizations of random vector $\xi$ on its support $\Xi$, and $\Xi^{|\Omega|}:=\{\xi^1, \xi^2,\dots,\xi^{|\Omega|}\}$ be the corresponding discretization of $\Xi$, 
where $\xi^j\neq \xi^j$ if $i\neq j$. Let $r_{ij}^\omega$ denote the amount of resource $i$ required by user $j$ under scenario $\omega$, and let $\mu_j^\omega$ denote the maximum share of a resource required by user $j$ to process a job under scenario $\omega$. Then
\vspace{-2mm}
\begin{equation}\label{small_f_def}
F(x,\xi^\omega) = \text{sign}(1-\beta)\left(\sum_{j=1}^{d}\left(\frac{\mu_{j}^\omega x_{j}}{\sum_{k=1}^{d} \mu_{k}^\omega x_{k}}\right)^{1-\beta}\right)^{\frac{1}{\beta}}\left(\sum_{j=1}^{d} \mu_{j}^\omega x_{j}\right)^{\lambda}.
\end{equation}
Let $\mathbb{P}^{|\Omega|}$ be a discrete measure assigning probability $p^\omega$ to each scenario $\xi^{\omega}$. Among the various approaches to handling chance constraints, for example, the big-M reformulation \cite{song2014chance}, and the bilinear reformulation \cite{wang2021chance}. We adopt the big-M reformulation here, while alternative methods are discussed in Section~\ref{sec:conclusion}. We write an approximation of the constraint in \eqref{dro_model} as:
\begin{equation}\label{scenario_chance_constraint}
    \sum_{\omega=1}^{|\Omega|} p^\omega\mathbbm{1}_{(0, \infty)}\left(c_i-\sum_{j=1}^dr_{ij}^\omega x_j\right)\geq \theta \quad \forall i=1,..,m.
\end{equation}
 If $\sum_{j =1}^d r_{ij}^\omega x_j \leq c_i$,  $\mathbbm{1}_{(0, \infty)}\left(c_i-\sum_{j=1}^dr_{ij}^\omega x_j\right)=1$, otherwise $\mathbbm{1}_{(0, \infty)}\left(c_i-\sum_{j=1}^dr_{ij}^\omega x_j\right)=0$. We introduce a new binary variable to represent $\mathbbm{1}_{(0, \infty)}\left(c_i-\sum_{j=1}^dr_{ij}^\omega x_j\right)$. Let $z_i^\omega$ be a binary variable such that $z_i^\omega = 1$ if $\sum_{j=1}^d r_{ij}^\omega x_j \leq c_i$, and $0$ otherwise.

\noindent Let $M^\omega_i$ be a sufficiently large number. The chance constraints of \eqref{dro_model} are approximated as 
\vspace{-2mm}
\begin{align}
\sum_{j=1}^dr_{ij}^\omega x_j+(M^\omega_i-c_i)z_{i}^\omega&\leq M^\omega_i\; \forall i=1,..,m, \forall \omega=1,..,|\Omega|,\label{big_M_original} \\[-1em]
\sum_{\omega=1}^{|\Omega|}p^\omega z_{i}^\omega&\geq \theta\quad \forall i=1,..,m.
\end{align}

\noindent Therefore, \eqref{dro_model} can be written as:
 \vspace{-2mm}
\begin{equation}\label{dro_model_sampled}
\begin{aligned}
\max_{x\in\mathcal{X},z} \min_{\mathbb{P}^{|\Omega|}\in \mathcal{P}^{|\Omega|}}\; &\sum_{\omega=1}^{|\Omega|} p^\omega F(x,\xi^\omega)\\
\text{s.t.} \quad&
\sum_{j=1}^dr_{ij}^\omega x_j+\left(M^\omega_i-c_i\right) z_{i}^\omega \leq M^\omega_i \quad \forall i=1,..,m, \forall \omega=1,..,|\Omega|,\\
\qquad \quad&  \sum_{\omega=1}^{|\Omega|}p^\omega z_{i}^\omega\geq \theta\quad \forall i=1,..,m,
\\
\qquad \quad &z_{i}^\omega\in\{0,1\} \quad\forall i=1,..,m,\forall \omega=1,..,|\Omega|.
\end{aligned}
\end{equation}

\noindent An approximation of the ambiguity set $\mathcal{P}^{|\Omega|}$ is given as: \vspace{-2mm}
\begin{equation}\label{ambiguity_dro_sampled}
    \begin{aligned}
 &\mathcal{P}^{|\Omega|}=\{\mathbb{P}^{|\Omega|}\mid\mathbb{P}^{|\Omega|}=(p^1,\dots,p^{|\Omega|}),\sum_{\omega=1}^{|\Omega|} p^\omega=1,p^\omega\geq 0 \;\forall \omega, \\
 &\underline{u}_{ij}\leq \sum_{\omega=1}^{|\Omega|}p^\omega r_{ij}^\omega\leq  \bar{u}_{ij}, \sum_{\omega=1}^{|\Omega|}p^\omega\left( r_{ij}^\omega\right)^2-\left(\sum_{\omega=1}^{|\Omega|}p^\omega r_{ij}^\omega\right)^2\leq  \bar{\sigma}_{ij}^2,\forall i,j\}.      \end{aligned}
\end{equation}

In the following result, we show that the optimal objective function and its fairness component under the SA-DR model \eqref{dro_model_sampled} also change similarly to that of the DR model \eqref{dro_model} with respect to changes in the fairness parameter $\beta$.

\begin{corollary} \label{cor:SADR_beta}
    Let $\Xi^{|\Omega|}=\{\xi^1,\ldots,\xi^{|\Omega|}\}$ be the scenario set with probabilities \(\mathbb{P}^{|\Omega|}=\{p^1,\ldots,p^{|\Omega|}\}\) constrained by \eqref{ambiguity_dro_sampled}, and let \(z_i^\omega\) enforce the chance constraints as in \eqref{big_M_original}-\eqref{dro_model_sampled}. Define the sampled value functions \vspace{-2mm}

\[
v^{|\Omega|}_{\mathrm{fair}}(\beta):=\max_{x\in\mathcal{X}}
\ \min_{\mathbb{P}^{|\Omega|}\in\mathcal{P}^{|\Omega|}(x)}\ \sum_{\omega=1}^{|\Omega|}p^\omega F_\beta(x,\xi^\omega),\quad 
v^{|\Omega|}(\beta):=\max_{x\in\mathcal{X}}
\ \min_{\mathbb{P}^{|\Omega|}\in\mathcal{P}^{|\Omega|}(x)}\ \sum_{\omega=1}^{|\Omega|}p^\omega F(x,\xi^\omega),
\]
where \(\mathcal{P}^{|\Omega|}(x)\) is the set of sampled distributions satisfying the linearized chance constraints for \(x\). Then, for \(\beta>1\), both \(v^{|\Omega|}_{\mathrm{fair}}(\beta)\) and \(v^{|\Omega|}(\beta)\) are nonincreasing in \(\beta\).
\end{corollary}

\noindent \textbf{\textit{Convergence of SA-DR Model to DR Model.}} Under the following assumption, we show the convergence of the SA-DR model \eqref{dro_model_sampled} to the DR model \eqref{dro_model} in Appendix~\ref{convergence_proof}.
\begin{assumption}\label{feasible_assumption}
\textit{We assume that for every $\mathbb{P}\in\mathcal{P}$, the probability density $f_\mathbb{P}(\cdot)$ is bounded by $C^\mathcal{P}$ and there exists a $\mathbb{P}_0$ such that $\underline{u}_{ij}\leq  \mathbb{E}_{\mathbb{P}_0}[\tilde{r}_{ij}]\leq   \bar{u}_{ij}, \text{var}_{\mathbb{P}_0}(\tilde{r}_{ij})\leq  \bar{\sigma}_{ij}^2,\forall i,j$. $\mathcal{X}$ can be written as a set of deterministic inequalities $g(x)\leq 0$, where $g(\cdot)$ is a vector. For the optimal solution $\mathbb{P}^*$ of \eqref{dro_model} and any feasible solution $x$ for a given $\theta$, the partial derivative, $\partial\Big(\text{Prob}_{\mathbb{P}^*}\left(\sum_{j=1}^{d} \tilde{r}_{ij} x_{j} \leq c_{i}\right)\Big)/\partial x,\forall i$, exists and $\exists h\in\mathbb{R}^d$ such that $\theta-\text{Prob}_{\mathbb{P}^*}\left(\sum_{j=1}^{d} \tilde{r}_{ij} x_{j} \leq c_{i}\right)-\Big(\partial\Big(\text{Prob}_{\mathbb{P}^*}\left(\sum_{j=1}^{d} \tilde{r}_{ij} x_{j} \leq c_{i}\right)\Big)/\partial x\Big)h<0, \forall i$, satisfying $g(x)+(\partial g(x)/\partial x)h<0$.
}  
\end{assumption}

\section{Stochastic Fairness Properties}\label{subsec:properties}
There are three important properties of the deterministic FDS model: Pareto-efficiency, sharing incentive, and envy-freeness. For the stochastic model, we now define and prove these three properties. The definitions reduce to the deterministic case when $\tilde{r}_{ij}$ is deterministic. Proofs of all propositions in this section are provided in Appendix~\ref{proof:properties}.\vspace{0mm}

\textbf{\textit{Stochastic Pareto-Efficiency.}} We provide a definition of the stochastic Pareto-efficiency based on \cite[\textit{Definition 2}]{Joe13}.

\begin{definition} \label{def2}
\textit{A function $f(x):=\mathbb{E}\left[F(x,\xi)\right]$ is stochastic Pareto-efficient if whenever $x$ Pareto-dominates $y$, $\hat{f}^{|\Omega|}(x)>\hat{f}^{|\Omega|}(y)$, where $\hat{f}^{|\Omega|}(\cdot)$ is any scenario based evaluation of $f(\cdot)$. }
\end{definition}

Using Definition \ref{def2}, the following proposition shows the stochastic Pareto-efficient property.
\begin{proposition}\label{prop:stochastic_pareto}
For any given ambiguity set $\mathcal{P}^{|\Omega|}$,  \eqref{small_f_def} is stochastic Pareto-efficient if $|\lambda|\geq |(1-\beta)/\beta|$.
\end{proposition}

\textbf{\textit{Stochastic Sharing Incentive.}} We now present the definition and parameter conditions for stochastic sharing incentive based on \cite[\textit{Definition 3}]{Joe13}. 

\begin{definition} \label{def4}
\textit{Let $\tilde{r}_{ij}$ be the amount of resource $i$ required by user $j$ and let $c_i$ be the total capacity of resource $i$. Let $\tilde{\eta}_{ij}=\frac{\tilde{r}_{ij}}{c_i}$ denote the share of resource $i$ required by user $j$ to process a job, and let $\tilde{\mu}_j = \max_i{\tilde{\eta}_{ij}}$ denote the maximum share of any resource required by user $j$ to process a job. Let $x_j$ be the number of jobs allocated to user $j$. Stochastic sharing incentive is the property that the dominant share satisfies 
$\sum_{\omega=1}^{|\Omega|} \hat{p}^\omega\left(\mu_{j}^\omega x_j\right)\geq 1/{d}$ for all $j=1,.., d$,
where  $\hat{p}^\omega$ and $x_j$ are the optimal solutions of \eqref{dro_model_sampled} and $\mu_{j}^\omega$ is the scenario evaluation of $\tilde{\mu}_j$ and $\mu_{j}^\omega x_{j}$ is the  the maximum share of user $j$ in scenario $\omega$. }
\end{definition}

Using Definition \ref{def4}, the following result shows the stochastic sharing incentive property under certain conditions. 

\begin{proposition}\label{stochastic_sharingincentive_prop}
For any given ambiguity set $\mathcal{P}^{|\Omega|}$, \eqref{dro_model_sampled} satisfies the stochastic sharing incentive property when 
 $\beta>1$ and $\lambda=(1-\beta)/\beta$. 
\end{proposition}


\textbf{\textit{Stochastic Envy-Freeness.}} We now provide a definition and parameter conditions for stochastic envy-freeness based on \cite[Definition 4]{Joe13}.
\begin{definition}\label{stochastic_ParetoEfficiency}
\textit{User $j$ envies user $k$ if $\eta_{ik}^\omega x_k \geq \eta_{ij}^\omega x_j$ for all resources $i$, with at least one strict inequality, for all scenarios $\omega$. Here, $\eta_{ik}^\omega$ and $\eta_{ij}^\omega$ are the scenario evaluation of $\tilde{\eta}_{ik}$ and $\tilde{\eta}_{ij}$, respectively.
Stochastic envy-freeness holds if no user envies another user's allocation. }
\end{definition}

Using Definition \ref{stochastic_ParetoEfficiency}, we show the stochastic envy-free property in the following result.
\begin{proposition}\label{prop:stochastic_envyfree}
For any given ambiguity set  $\mathcal{P}^{|\Omega|}$, \eqref{dro_model_sampled} satisfies the stochastic envy-free property when $\lambda=(1-\beta)/\beta$ and  $\beta>1$.
\vspace{-2mm}
\end{proposition}

\section{A Cutting Surface Algorithm}\label{sec:algorithm}
In this section, we provide a cutting surface algorithm to solve the proposed SA-DR model \eqref{dro_model_sampled} to $\epsilon$-accuracy within finite iterations.
Let $q(x)=\min_{\mathbb{P}\in \mathcal{P}(x)}\mathbb{E}_{\mathbb{P}}[F(x,\xi)] $ and $\hat{v}=\max_{x\in\mathcal{X}} q(x)$, $q^{|\Omega|}(x)=\min_{\mathbb{P}^{|\Omega|}\in \mathcal{P}^{|\Omega|}(x)}\sum_{\omega=1}^{|\Omega|} p^\omega F(x,\xi^\omega) $ and $\hat{v}^{|\Omega|}=\max_{x\in\mathcal{X}} q^{|\Omega|}(x)$.  For a given solution $\hat{x}$, the inner problem of the proposed SA-DR model \eqref{dro_model_sampled} is 

\vspace{-2mm}
\begin{equation}\label{primal}
\begin{aligned}
\min_{\mathbb{P}^{|\Omega|}} \;& \sum_{\omega=1}^{|\Omega|}p^\omega F(\hat{x},\xi^\omega)\\
s.t.\; 
&\sum_{\omega=1}^{|\Omega|} p^\omega r_{ij}^\omega \leq \bar{u}_{ij} \quad \forall i=1,..,m,j=1,..,d,\\
&\sum_{\omega=1}^{|\Omega|} p^\omega r_{ij}^\omega \geq \underline{u}_{ij} \quad \forall i=1,..,m,j=1,..,d,\\
&\sum_{\omega=1}^{|\Omega|} p^\omega (r_{ij}^\omega)^2 -\left(\sum_{\omega=1}^{|\Omega|} p^\omega r_{ij}^\omega\right)^2\leq \bar{\sigma}_{ij}^2 \; \forall i=1,..,m, j=1,..,d,\\
& \text{Prob}_{\mathbb{P}^{|\Omega|}}\left(\sum_{j=1}^{d} r_{ij}^\omega \hat{x}_{j} \leq c_{i}\right)\geq  \theta\quad  \forall i=1,..,m,\\
& \sum_{\omega=1}^{|\Omega|} p^\omega = 1, p^\omega \geq 0 \quad \forall \omega=1,..,|\Omega|.
\end{aligned}
\end{equation}

We define the feasible solution set of \eqref{primal} to be $\bar{\mathcal{P}}^{|\Omega|}$. We now introduce a new variable $\lambda$ and rewrite the SA-DR model \eqref{dro_model_sampled} as:
\vspace{-2mm}
\begin{equation}\label{general}
\begin{aligned}
\max_{\lambda,x} &\quad \lambda\\
s.t. & \min_{\mathbb{P}^{|\Omega|}\in\bar{\mathcal{P}}^{|\Omega|}}\sum_{\omega=1}^{|\Omega|}p^\omega F(x,\xi^\omega) \geq \lambda, \quad  x\in \mathcal{X}'(\mathbb{P}^{|\Omega|}),
\end{aligned}
\end{equation}
where $\mathcal{X}'(\mathbb{P}^{|\Omega|})=\Big\{x\in\mathcal{X}\Big|\text{Prob}_{\mathbb{P}^{|\Omega|}}\big(\sum_{j=1}^{d} r_{ij}^\omega x_{j} \leq c_{i}\big)\geq  \theta, \forall i=1,..,m\Big\}.$ 
Clearly, the model \eqref{general} can be reformulated to an equivalent master problem as follows:
\begin{equation}\label{joint}
\begin{aligned}
\max_{\lambda,x} &\; \lambda\\
s.t.& -\lambda+\sum_{\omega=1}^{|\Omega|}p^\omega F(x,\xi^\omega) \geq 0 \quad x\in\mathcal{X}'(\mathbb{P}^{|\Omega|}), \forall \;\mathbb{P}^{|\Omega|}\in\bar{\mathcal{P}}^{|\Omega|}.\\
\end{aligned}
\end{equation}
The separation problem for identifying a violated constraint at the solution $\hat{x}$ of the current master problem is written as: 
\vspace{-5mm}
\begin{equation}    \min_{\mathbb{P}^{|\Omega|}\in\bar{\mathcal{P}}^{|\Omega|}} g(\hat{x},\mathbb{P}^{|\Omega|}):=  -\hat{\lambda}+\sum_{\omega=1}^{|\Omega|}p^\omega F(\hat{x},\xi^\omega).
\end{equation}

In the following definition, we also define the $\epsilon$-feasible solution and $\epsilon$-optimal solution.
\begin{definition}
\textit{For a given solution $\hat{x}$,  $\mathbb{P}_0^{|\Omega|}$ is an $\epsilon-$feasible solution of the inner problem \eqref{primal} if $\mathbb{P}_0^{|\Omega|}=\{p_0^1,\dots,p_0^{|\Omega|}\}, \underline{u}_{ij}-\epsilon\leq  \sum_{\omega=1}^{|\Omega|}p_0^\omega r_{ij}^\omega\leq \bar{u}_{ij}+\epsilon, \sum_{\omega=1}^{|\Omega|} p^\omega (r_{ij}^\omega)^2 -\left(\sum_{\omega=1}^{|\Omega|} p^\omega r_{ij}^\omega\right)^2\leq \bar{\sigma}_{ij}^2+\epsilon ,\forall i,j$.
For a given solution $\hat{x}$, $\mathbb{P}_0^{|\Omega|}$ is an $\epsilon-$optimal solution if $\mathbb{P}_0^{|\Omega|}$ is an $\epsilon-$feasible solution of the inner problem \eqref{primal} and $q^{|\Omega|}(\hat{x})\leq q(\hat{x})+\epsilon$. 
A point $x_0$ is an $\epsilon$-feasible solution of the master problem \eqref{joint} if $\min_{\mathbb{P}^{|\Omega|}\in\bar{\mathcal{P}}^{|\Omega|}} g(x_0, \mathbb{P}^{|\Omega|}) \geq -\epsilon$. A point $x_0$ is an $\epsilon$-optimal solution if $x_0$ is a feasible solution of the master problem \eqref{joint} and $\lambda_0 \geq \operatorname{Val}\eqref{joint}-\epsilon$.  Similarly, a point $x_0$ is an $\epsilon$-optimal solution if $x_0$ is a feasible solution of the proposed SA-DR model \eqref{dro_model_sampled} and $q^{|\Omega|}(x_0) \geq \hat{v}^{|\Omega|}-\epsilon$.}  
\end{definition} 

We show that the inner problem \eqref{primal} can be solved to $\epsilon$-optimality in Appendix~\ref{optimality_proof} and the master problem \eqref{joint} can be solved to $\epsilon$-optimality using a cutting surface algorithm  \cite{luo2021finitely}. Algorithm~\ref{alg:probcuts} presents pseudocode for the cutting surface algorithm, which returns an $\epsilon$-optimal solution to the master problem \eqref{joint} in a finite number of iterations.

\begin{algorithm}
\caption{The Cutting Surface Algorithm}\label{alg:probcuts}
{\bf Initialize:} $\Omega_0= \emptyset$,  $t=0$.\\
 {\bf Step 1:} Find an optimal solution $(x_t,\lambda_t)$ of the problem $\max_{\lambda,x} \{\lambda\;\; s.t.\; g(x,\mathbb{P}^{|\Omega|})\geq 0, x\in\mathcal{X}', \; \forall \mathbb{P}^{|\Omega|}\in \bar{\mathcal{P}}_t^{|\Omega|}\}$.\\
 {\bf Step 2:} Find a $\epsilon/2$-optimal solution  $\mathbb{P}^{|\Omega|}_{t+1}$ of the problem $\min_{\mathbb{P}^{|\Omega|}\in\mathcal{P}^{|\Omega|}}g(x_t,\mathbb{P}^{|\Omega|})$. If $g(x_t,\mathbb{P}^{|\Omega|}_{t+1})\geq -\epsilon/2$, stop and return $x_t$ and $\lambda_t$. Otherwise, let $\bar{\mathcal{P}}_{t+1}^{|\Omega|}=\bar{\mathcal{P}}_t^{|\Omega|} \cup \{\mathbb{P}^{|\Omega|}_{t+1}\}$, $t=t+1$ and go to Step 1. 
\end{algorithm}

\begin{theorem}\label{thm:alg}
Suppose we have an oracle that generates an optimal solution of the problem $\max_{\lambda,x} \{\lambda\;\; s.t.\; g(x,\mathbb{P}^{|\Omega|})\geq 0,x\in\mathcal{X}'(\mathbb{P}^{|\Omega|}),\; \forall \mathbb{P}^{|\Omega|}\in \bar{\mathcal{P}}_t^{|\Omega|}\}$ for any $\bar{\mathcal{P}}^{|\Omega|}_t\subseteq \bar{\mathcal{P}}^{|\Omega|}$ and an oracle that generates an $\epsilon$-optimal solution of the problem $\min_{\mathbb{P}^{|\Omega|}\in\mathcal{P}^{|\Omega|}}g(x,\mathbb{P}^{|\Omega|})$  for any $x$ and $\epsilon>0$. The set $\mathcal{X}$ and the support $\Xi^{|\Omega|}$ are compact. The function $F(\cdot,\cdot)$ is bounded on $\mathcal{X}\times \Xi^{|\Omega|}$. Algorithm~\ref{alg:probcuts} terminates in finitely many iterations and returns an $\epsilon$-feasible solution of the master problem \eqref{joint} such that $\lambda \geq \operatorname{Val}\eqref{joint}-\epsilon$ at termination. Algorithm~\ref{alg:probcuts} returns a solution $(\hat{x})$ which is an $\epsilon-$optimal solution of the proposed SA-DR model \eqref{dro_model_sampled}. 
\end{theorem}
The proof of Theorem~\ref{thm:alg} can be found in Appendix~\ref{proof:algorithm}.

\section{Computational Results}\label{sec:Computational_Results}
\subsection{Alternative Models}

In this section, we first present alternative models under different levels of information available about resource requirement distribution. We then conduct a sensitivity analysis of resource allocation decisions with respect to various model parameters. Finally, we identify the value of information by comparing stochastic resource allocation decisions under varying levels of information about the resource requirement distributions. These models include the expected value-based deterministic model \eqref{fairness_stoch_obj_alter}, the worst-case information-based robust model \eqref{fairness_stoch_obj_ro}, and the full information-based sample average approximation (SAA) model \eqref{SAA_model}. The deterministic expected value-based model \eqref{fairness_stoch_obj_alter} simplifies the problem by relying only on the expected value in the FDS function. The robust model \eqref{fairness_stoch_obj_ro} adopts a worst-case approach, optimizing fairness without distribution information, but assuming that its support is known. Finally, the SAA model \eqref{SAA_model} assumes full information on the resource requirement distribution.


\textbf{\textit{Expected Value-Based Deterministic Model.}} We define an expected value-based deterministic model as follows: 
\begin{equation}\label{fairness_stoch_obj_alter}
   \begin{aligned}
\max_{x\in\mathcal{X}} \; & F(x,\mathbb{E}_{\mathbb{P}}\left[\xi\right])\\
\text{s.t.}\; &\text{Prob}_{\mathbb{P}}\left(\sum_{j=1}^{d} \tilde{r}_{ij} x_{j} \leq c_{i}\right)\geq 1-\theta\; \forall i=1,..,m.
\end{aligned} 
\end{equation}

For a given $\mathbb{P}$, the above model simplifies the stochastic fairness model to be deterministic because $\mathbb{E}_{\mathbb{P}}\left[\xi\right]$ is not random. As it only focuses on one value $\mathbb{E}_{\mathbb{P}}\left[\xi\right]$, it will lose the information on the randomness of the entire fairness function, which we will show in the following sections. 

\textbf{\textit{Robust Fairness Model.}} A robust model can be defined as: 
\begin{equation}\label{fairness_stoch_obj_ro}
   \begin{aligned}
\max_{x\in\mathcal{X}} \min_{\xi^\omega\in\Xi^{|\Omega|}}\; & F(x,\xi^\omega)\\
\text{s.t.}\; &\max_{\xi^\omega\in\Xi^{|\Omega|}}\sum_{j=1}^{d} r^\omega_{ij} x_{j} \leq c_{i}\quad \forall i=1,..,m.
\end{aligned} 
\end{equation}

In this model, we consider the worst-case scenario for both the objective function and the constraints. When the ambiguity set of probability distributions contains all probability distributions on the support of the uncertain parameters and chance probability $\theta=1$, the proposed partial information-based SA-DR model \eqref{dro_model_sampled} and the robust model \eqref{fairness_stoch_obj_ro} are equivalent.

\textbf{\textit{Sample Average Stochastic Fairness Model.}} We apply a sample average approximation (SAA) approach to represent the true distribution with a finite number of scenarios, following \cite{Shapiro03}. Specifically, instead of using uniformly distributed scenarios, the support $\Xi$ is discretized according to the true distribution. Although this model is less practical in the current context, it will be used to identify the value of information. The SAA model can be defined as:
\begin{equation}\label{SAA_model}
\begin{aligned}
\max_{x\in\mathcal{X},z} \; &
\frac{1}{|\Omega|}\sum_{\omega=1}^{|\Omega|} F(x,\xi^\omega)\\
 \text{s.t.} \;&
\sum_{j=1}^dr_{ij}^\omega x_j+(M_{i}^{\omega}-c_i)z_{i}^\omega\leq M_i^{\omega}\; \forall i=1,..,m, \forall \omega=1,..,|\Omega|,\\
\qquad \quad&  \frac{1}{|\Omega|}\sum_{\omega=1}^{|\Omega|}z_{i}^\omega\geq \theta\qquad \forall i=1,..,m,\\
\qquad \quad &z_{i}^\omega\in\{0,1\} \quad\forall i=1,..,m,\forall \omega=1,...,|\Omega|.
\end{aligned}
\end{equation}

\subsection{Numerical Analysis}\label{subsec:numerical_test}

We first examine the impact of chance probability and resource requirement variance on allocation decisions, fairness, and efficiency when full information about the resource requirement distribution is available. Specifically, we analyze the expected value-based deterministic model \eqref{fairness_stoch_obj_alter} and the full information-based SAA model \eqref{SAA_model}, assuming that the required resources follow a Bernoulli distribution. Additionally, we use the Azure Public Dataset \cite{Cortez17} to evaluate the performance of the SAA model \eqref{SAA_model} under various types of resource requirement distributions. To explore the impact of resource requirement variance in a partial-information setting, we employ our proposed SA-DR model, which utilizes mean and variance information \eqref{dro_model_sampled}, and assess its performance by varying the size of the ambiguity set on the Azure Public Dataset. Finally, we derive insights into the performance of our partial information-based SA-DR model relative to other models as the size of the demand uncertainty set changes under different levels of information availability. In particular, we utilize simulated data with box demand uncertainty of varying support sizes, generated from the CloudSim study \cite{Calheiros11}, to compare the performance of the partial information-based SA-DR model \eqref{dro_model_sampled}, the worst-case information-based robust model \eqref{fairness_stoch_obj_ro}, the expected value-based deterministic model \eqref{fairness_stoch_obj_alter}, and the full information-based SAA model \eqref{SAA_model}.

The computation is implemented in the Gekko Python platform \cite{Beal18} using the APOPT solver. A laptop using the Windows operating system with an Intel(R) 2.50 GHz processor and 8 GB RAM is used for computations. 
\subsubsection{\textbf{Chance Probability}}\label{subsec:significanr_level}
We consider a toy example of a datacenter with CPU and RAM requirement constraints. Assume there are two users, each requiring a fixed amount of each resource to accomplish a job.
We assume jobs are infinitely divisible and use the same coefficients as in \cite{Ghodsi11} and \cite{Joe13} for benchmark performance.
Let $x_1$ be the number of jobs allocated to User $1$ and $x_2$ be the number of jobs allocated to User $2$.  
Let $\tilde{\mu}_{1}$ and $\tilde{\mu}_{2}$ be the maximum share of the resource required by User $1$ and User $2$ to process a job. We can write the objective function of the deterministic expected value-based model \eqref{fairness_stoch_obj_alter} and the SAA model \eqref{SAA_model} as: 
\begin{equation}
\begin{aligned}
 &F(x,\mathbb{E}[\xi])
    = \quad \text{sign}(1-\beta) \left(\Big(\frac{\mathbb{E}\left[\tilde{\mu}_1\right] x_{1}}{\mathbb{E}\left[\tilde{\mu}_1\right]x_1+\mathbb{E}\left[\tilde{\mu}_2\right] x_{2}}\right)^{1-\beta}+\left(\frac{\mathbb{E}\left[\tilde{\mu}_2\right] x_{2}}{\mathbb{E}\left[\tilde{\mu}_1\right]x_1+\mathbb{E}\left[\tilde{\mu}_2\right] x_{2}}\right)^{1-\beta}\Big)^{\frac{1}{\beta}}\left(\mathbb{E}\left[\tilde{\mu}_1\right]x_1+\mathbb{E}\left[\tilde{\mu}_2\right] x_{2}\right)^{\lambda}.
    \end{aligned}
\end{equation}
\begin{equation}
\begin{aligned}
   & \mathbb{E}[F(x,\xi)]= \text{sign}(1-\beta) \mathbb{E}\Bigg[  \Bigg(\left(\frac{\tilde{\mu}_1x_{1}}{\tilde{\mu}_1x_1+\tilde{\mu}_2 x_{2}}\right)^{1-\beta}+\left(\frac{\tilde{\mu}_2 x_{2}}{\tilde{\mu}_1x_1+\tilde{\mu}_2x_{2}}\right)^{1-\beta}\Bigg)^{\frac{1}{\beta}}\left(\tilde{\mu}_1x_1+\tilde{\mu}_2 x_{2}\right)^{\lambda}\Bigg].
    \end{aligned}
\end{equation}
We set parameters $\beta=2$ and $\lambda=(1-\beta)/\beta=-1/2$ to ensure that the models satisfy all three stochastic fairness properties in Section~\ref{subsec:properties}.
To examine the impact of changing the chance probability $\theta$ on the resource allocations and the objective function, we consider two models:
\begin{equation}
\begin{aligned}\label{test:significant_level_in}
\max \quad &F(x,\mathbb{E}[\xi])\\
\text{s.t.}\quad &\text{Prob}(x_1+cx_2\leq 9)\geq \theta, \\
&  \text{Prob}(4x_1+dx_2\leq18)\geq \theta.
\vspace{-2mm}
\end{aligned}
\end{equation}
\begin{equation}
\begin{aligned}\label{test:significant_level_out}
\max\quad & \mathbb{E}[F(x,\xi)]\\
\text{s.t.}\quad &\text{Prob}(x_1+cx_2\leq 9)\geq \theta, \\
& \text{Prob}(4x_1+dx_2\leq18)\geq \theta.
\end{aligned}
\end{equation}
Let the resource requirement of User $1$ be deterministic, and the number of resources $c$ and $d$ required by User $2$ be random variables. Assume  
{\small $ \begin{pmatrix}
    c\\
    d
  \end{pmatrix}$}
  be a vector and follow the Bernoulli distribution with the probability $\mathbb{P} =0.5$. We choose the value of $c$ and $d$ with a large variance: {\small $
 \begin{pmatrix}
    c\\
    d
  \end{pmatrix}=\left(
  \begin{pmatrix}
    0.5\\
    0.5
  \end{pmatrix},
  \begin{pmatrix}
    5.5\\
    1.5
  \end{pmatrix}\right)$}. 
Let the chance probability be a parameter $\theta\in [0.90,0.92,...,0.98,1]$.
The left and right columns of Figure~\ref{fig:significant_level} show the impact of chance probability ($\theta$) on the allocations ($x_1,x_2$) and the optimal objective value from the expected value-based deterministic model \eqref{test:significant_level_in}, and the full information-based sample average model \eqref{test:significant_level_out}, respectively.

\begin{figure}[htbp]
\begin{multicols}{2}
    \includegraphics[width=\linewidth]{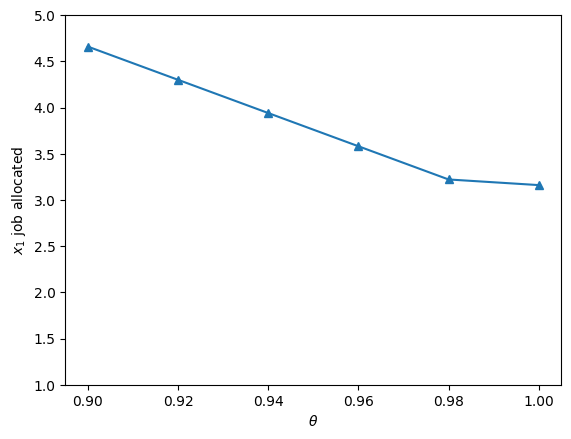}\par
    \includegraphics[width=\linewidth]{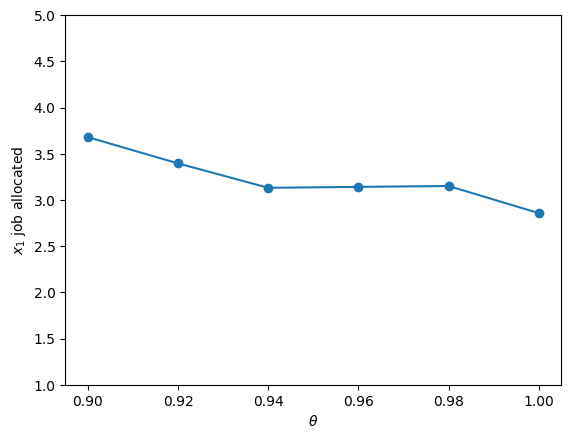}\par
\end{multicols} \vspace{-8mm} 
\begin{multicols}{2}
    \includegraphics[width=\linewidth]{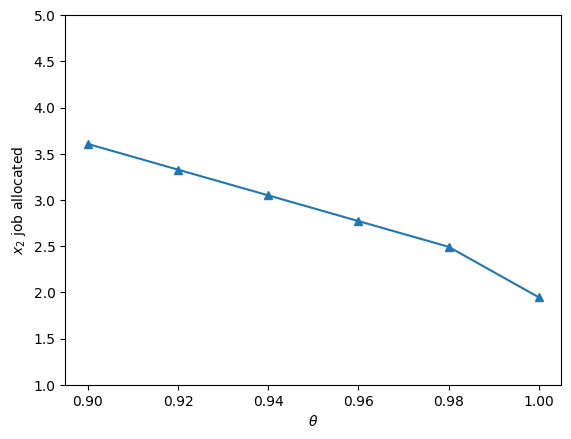}\par
        \includegraphics[width=\linewidth]{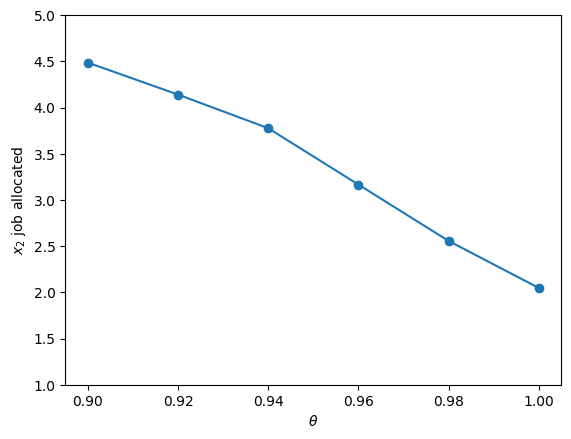}\par
\end{multicols} \vspace{-8mm} 
\begin{multicols}{2}
     \includegraphics[width=\linewidth]{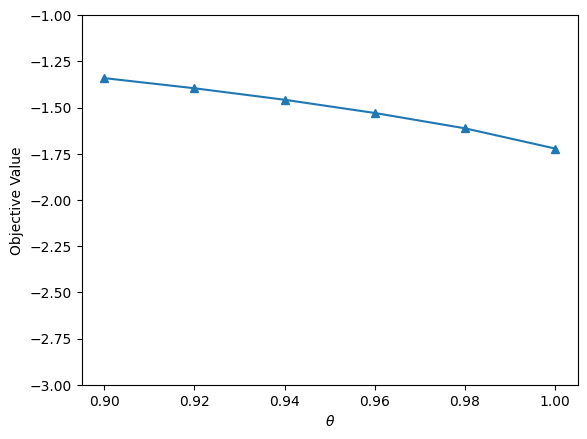}\par
     \includegraphics[width=\linewidth]{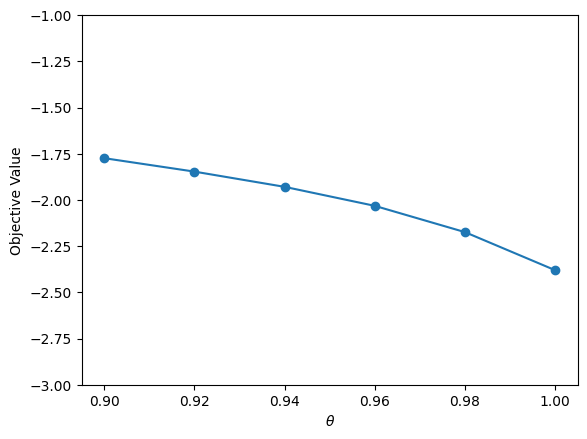}\par
    \end{multicols}
    \vspace{-8mm}
\caption{Job allocations ($x_1$ and $x_2$) and the optimal objective value under different chance probabilities for the deterministic expected value-based model in the left column and the sample average model in the right column.}
\label{fig:significant_level}
\end{figure}



\noindent Figure~\ref{fig:significant_level} illustrates that job allocations to User $1$ ($x_1$) and User $2$ ($x_2$) decrease as the chance probability ($\theta$) increases. Additionally, the optimal objective value---which represents both fairness and efficiency---also declines as the chance probability increases under both models. When the chance probability is small, it allows violations of the resource constraints in certain scenarios, thereby enabling more jobs to be allocated—particularly to a user with uncertain resource requirements (User $2$). This relaxation of the chance constraints results in higher fairness, as evidenced by an increase in the optimal objective value when $\theta$ decreases. However, this suggests that while a smaller chance probability can lead to a more equitable allocation of resources, it comes at the expense of a higher likelihood of resource constraint violations.

\subsubsection{\textbf{Resource Requirement Variance}}\label{subsec:variance}
We use the same parameter settings as in Section~\ref{subsec:significanr_level} to examine the impact of requirement variance on the resource allocations, considering the following two models: 
\begin{equation}
\begin{aligned}\label{test:cvariance_dvariance_in}
\max \quad & F(x,\mathbb{E}[\xi])\nonumber\\ 
\text{s.t.}\quad &\text{Prob}(x_1+cx_2\leq 9)\geq 0.95,\\
& \text{Prob}(4x_1+dx_2\leq 18)\geq 0.95.
\end{aligned}
\end{equation}
\begin{equation}
\begin{aligned}\label{test:cvariance_dvariance_out}
\max \quad & \mathbb{E}[F(x,\xi)]\\
\text{s.t.}\quad & \text{Prob}(x_1+cx_2\leq 9)\geq 0.95, \\
&\text{Prob}(4x_1+dx_2\leq 18)\geq 0.95.
\end{aligned}
\end{equation}
We now consider the Bernoulli distribution with the probability $\mathbb{P}\in [0.05,0.10,...,0.95]$,  
and the amount of resources being {\small $
 \begin{pmatrix}
    c\\
    d
  \end{pmatrix}=\left(
  \begin{pmatrix}
    0.5\\
    0.5
  \end{pmatrix},
  \begin{pmatrix}
    5.5\\
    1.5
  \end{pmatrix}\right),
  \left(
  \begin{pmatrix}
    1.5\\
    0.5
  \end{pmatrix},
  \begin{pmatrix}
    4.5\\
    1.5
  \end{pmatrix}\right),
  \left(
  \begin{pmatrix}
    2.5\\
    0.5
  \end{pmatrix},
  \begin{pmatrix}
    3.5\\
    1.5
  \end{pmatrix}\right)$}. The mean and variance of $c$ and $d$ vary as the probability $\mathbb{P}$ changes. 
The left and right columns of Figure~\ref{fig:two_variables} show the impact of resource requirement variance on the allocations and the optimal objective value from the expected value-based deterministic model \eqref{test:cvariance_dvariance_in}, and the full information-based sample average model \eqref{test:cvariance_dvariance_out}, respectively.

 \begin{figure}[htbp]
\begin{multicols}{2}
    \includegraphics[width=\linewidth]{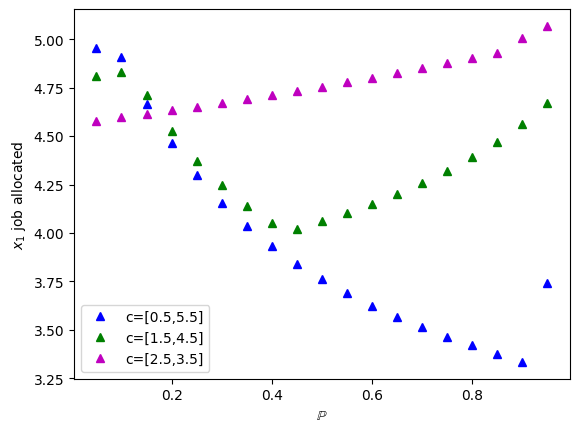}\par 
    \includegraphics[width=\linewidth]{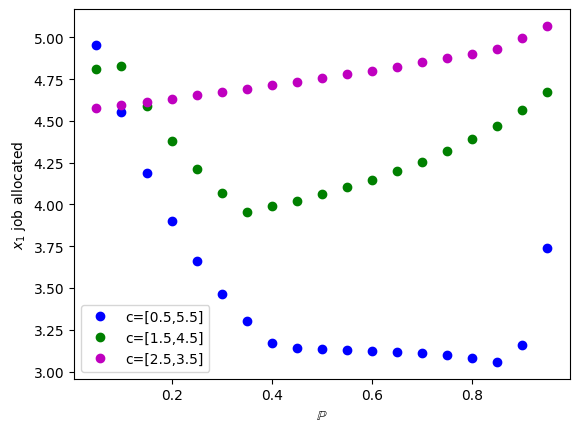}\par 
\end{multicols} \vspace{-8mm}
\begin{multicols}{2}
    \includegraphics[width=\linewidth]{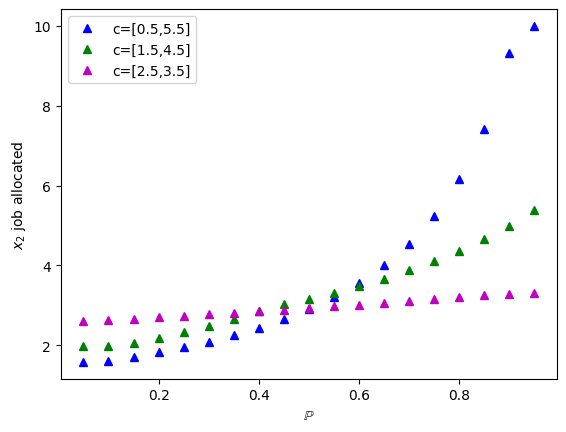}\par
    \includegraphics[width=\linewidth]{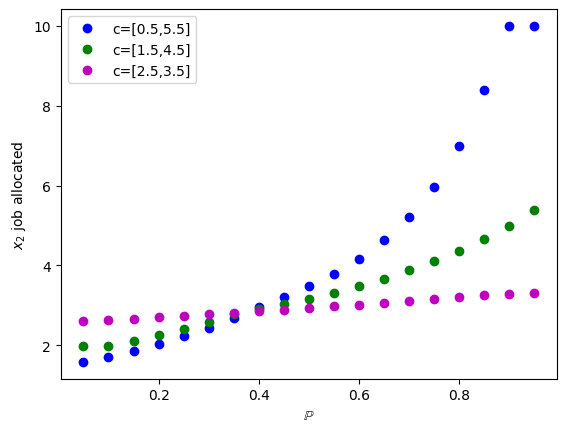}\par
\end{multicols} \vspace{-8mm}
\begin{multicols}{2}
    \includegraphics[width=\linewidth]{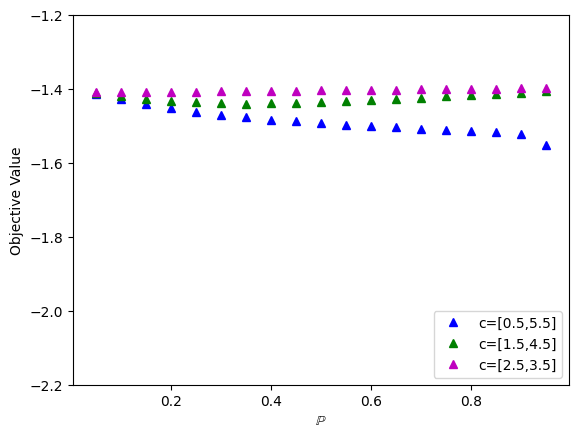}\par
    \includegraphics[width=\linewidth]{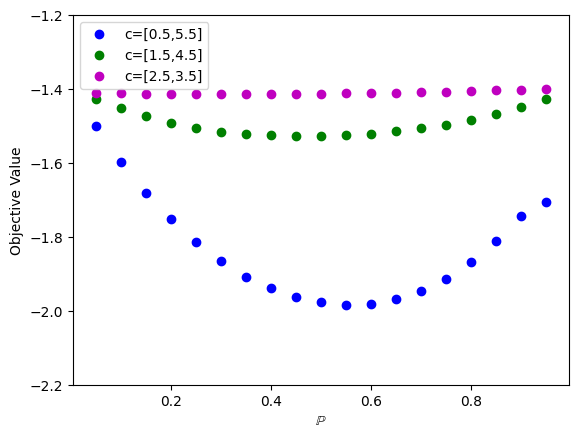}\par
    \end{multicols}
    \vspace{-8mm}
\caption{Job allocations and the optimal objective value under different variances of resource requirement $c$ and $d$ for the deterministic expected value-based model in the left column and the sample average model in the right column. }
\label{fig:two_variables}
\end{figure}

\noindent Figure~\ref{fig:two_variables} illustrates that the optimal objective value for the sample average model \eqref{test:cvariance_dvariance_out} initially decreases and subsequently increases, which indicates that the objective value decreases with an increase in the variance of the resource requirements. In contrast, the optimal objective value for the expected value-based deterministic model \eqref{test:cvariance_dvariance_in} remains relatively stable, as this model does not account for the variance of resource requirements. The increase in the variance of User $2$'s resource requirements causes the full information-based sample average model \eqref{test:cvariance_dvariance_out} to reduce allocations for the deterministic user (i.e., User $1$), resulting in lower fairness and efficiency compared to the expected value-based model \eqref{test:cvariance_dvariance_in}. This observation underscores the significant role that variance of resource requirements plays in allocation decisions, directly influencing fairness and efficiency, as well as the sensitivity of allocation decisions to changes in variance.

\subsubsection{\textbf{Type of Distribution}}\label{subsec:type_distribution}

We now study how the type of resource requirement distribution affects the allocation decisions.
  We use data from the Azure Public Dataset \cite{Cortez17}, which is used to deduce real consumption and market demand. Instead of assuming tasks are short and arriving at different times, we assume that at the decision time, the amount of resources needed to compute given tasks is uncertain. 
A random sample of 16 users is selected from the Azure dataset, and the capacities of the resources are $320$ cores, 6400 CPU, and 640 GB CPU memory. Each user has varying resource requirements for a job based on the dataset, and we assume jobs are infinitely divisible. The parameters are set to be $\beta=2$, $\lambda=(1-\beta)/\beta=-1/2$ and $\theta=0.95$. We consider $|\Omega|$ in the range of $25$ to $500$.

We assume that the resources follow two types of distributions---uniform and triangular---and we apply the full information-based SAA model \eqref{SAA_model}. We set the lower and upper limits for both the uniform and triangular distributions to the minimum and maximum resource requirements, respectively, and define the mode of the triangular distribution as the mean of these limits. We modified the calculation of the confidence interval bounds of min-max problems by \cite{Shapiro07} to max-min problems in our paper. The approximate $95 \%$ confidence lower bound for 
$\hat{f}^{|\Omega|}(\hat{x})$, with the feasible solution $\hat{x}$ is:
\begin{align}\label{upperbound}
L^{|\Omega|}(\hat{x})&:=\hat{f}^{|\Omega|}(\hat{x})-z_{0.05} \hat{\sigma}^{|\Omega|}(\hat{x}), \nonumber\\
(\hat{\sigma}^{|\Omega|}(\hat{x}))^2&:=\frac{1}{|\Omega|\left(|\Omega|-1\right)} \sum_{\omega=1}^{|\Omega|}\left[F\left(\hat{x}, \xi^\omega\right)-\hat{f}^{|\Omega|}(\hat{x})\right]^2,
\end{align}
and  $z_{0.05}=\Phi^{-1}(0.95)=1.64$ is the $0.05$-critical value of the standard normal distribution. For the upper bound, we denote the optimal value of the full information-based SAA model \eqref{SAA_model} by $\hat{v}^{|\Omega|}$. Next, we solve the stochastic models using independently generated replicates $M=5$ times, and denote the computed optimal values by $\hat{v}^{|\Omega|}_1,..., \hat{v}^{|\Omega|}_M$. 
Then, an approximate $95\%$ upper bound is given by
\begin{align}\label{lowerbound}
U^{|\Omega|, M}&:=\frac{1}{M} \sum_{m=1}^M \hat{v}^{|\Omega|}_m+t_{ 0.05, \nu} \hat{\sigma}^{|\Omega|, M},\nonumber\\
(\hat{\sigma}^{|\Omega|, M})^2&:=\frac{1}{M(M-1)} \sum_{m=1}^M\left(\hat{v}^{|\Omega|}_m-\frac{1}{M} \sum_{m=1}^M \hat{v}^{|\Omega|}_m\right)^2.
\end{align}
We have $\nu=M-1=4$, and $t_{0.05, \nu}=t_{0.05, 4}=2.13$ is the $0.05$-critical value of the $t$-distribution with $4$ degrees of freedom. Also, we define the relative gap as 
\begin{equation}\label{relativedifference}
    gap(|\Omega|)=\frac{|L^{|\Omega|}(\hat{x})-U^{|\Omega|, M}|}{|f(\hat{x})|}\times 100\%.
    \vspace{-2mm}
\end{equation}
\begin{figure}[htbp]
    \begin{center}
    \includegraphics[width=11cm]{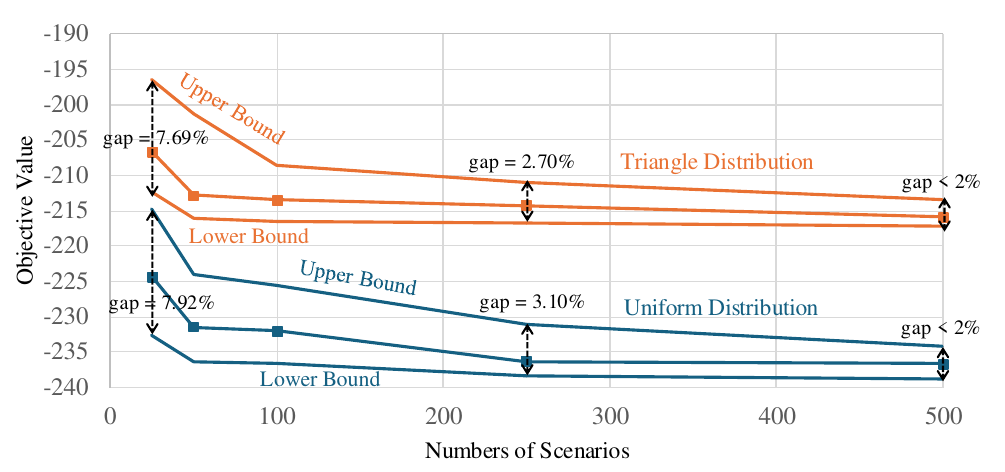}
    \vspace{-10mm}
    \caption{Impact of resource requirement distribution on the performance of the SAA model under the full-information setting.} \label{fig:scenario}
    \end{center}
\end{figure} 


\noindent Figure~\ref{fig:scenario} demonstrates that the triangular distribution consistently achieves a higher objective value than the uniform distribution across all scenario levels, primarily due to its smaller variance. Furthermore, as the number of scenarios increases, the objective value stabilizes, and the relative gap between the upper and lower bounds narrows for both distributions. This trend reflects improved solution precision and tighter confidence intervals, driven by the improved convergence of the full information-based SAA model \eqref{SAA_model} with a large number of scenarios.
\subsubsection{\textbf{Size of Ambiguity Set}}\label{subsec:size_ambiguity}
To analyze the impact of different ambiguity set sizes on resource allocation decisions, we use the same parameter settings as in Section~\ref{subsec:type_distribution}. However, we consider that the full information about the resource requirements distribution is unknown. In this case, we apply our SA-DR model \eqref{dro_model_sampled}, which relies only on partial information on the mean and variance of the resource requirement distribution. We define the ambiguity set as $\mathcal{P}^{|\Omega|}=\{p^\omega\mid\sum_{\omega=1}^{|\Omega|} p^\omega=1, p^\omega\geq 0 \; \forall \omega,(1-\Delta)u_{ij}\leq \sum_{\omega=1}^{|\Omega|}p^\omega r_{ij}^\omega\leq  (1+\Delta)u_{ij}, \sum_{\omega=1}^{|\Omega|}\left(p^\omega r_{ij}^\omega\right)^2-\left(\sum_{\omega=1}^{|\Omega|}p^\omega r_{ij}^\omega\right)^2\leq (1+\Delta)\sigma_{ij}^2,\forall i,j\}$, where we assume $\mu_{ij}$ is the sample mean and $\sigma_{ij}^2$ is the sample variance of the resource requirements. $\Delta$ is a parameter to represent the size of the ambiguity set. 
According to \cite{Lei24}, we modify the $95\%$ lower bound calculation as follows. 
Since $q^{|\Omega|}(x) =\min_{\mathbb{P}^{|\Omega|}\in \mathcal{P}^{|\Omega|}(x)}\;\mathbb{E}_{\mathbb{P}^{|\Omega|}}[F(x,\xi^\omega)]$,
for the feasible solution $\hat{x}$, the approximate $95\%$ lower bound is given by
\begin{equation}\label{upperbound_modified}
\begin{aligned}
   L^{|\Omega|,M}(\hat{x})&:=\frac{1}{M}\sum_{m=1}^M \hat{q}_{m}^{|\Omega|}(\hat{x})- t_{\alpha,\nu}\hat{\sigma}^{|\Omega|,M}(\hat{x}),\\
   (\hat{\sigma}^{|\Omega|,M}(\hat{x}))^2&:=\frac{1}{M\left(M-1\right)} \sum_{m=1}^{M}\left[\hat{q}_{m}^{|\Omega|}(\hat{x})-\frac{1}{M}\sum_{m=1}^M \hat{q}_{m}^{|\Omega|}(\hat{x})\right]^2
\end{aligned}
\end{equation}
where 
$\hat{q}_{m}^{|\Omega|}(\hat{x})$
is the computed optimal value of $q^{|\Omega|}(\hat{x})$ based on independently generated replicates
$M=5$ times.

We use the equations \eqref{lowerbound} and \eqref{relativedifference} to calculate the upper bound and the relative gap. 
\begin{table}[htbp]
    \centering
    \caption{Objective value for the proposed model and 95\% bounds with different sizes of the ambiguity set.}
    \begin{tabular}{|c|cccc|}
    \hline
     $\Delta$& obj   & LB & UB &  gap \\
     \hline
0.05&-15.50&-15.82&-15.31&3.29\%\\
\hline
0.1&-15.97&-16.35&-15.88&2.94\%\\
\hline
0.2&-16.41&-16.71&-16.35&2.15\%\\
\hline
0.5&-17.53&-17.56&-17.44&0.68\%\\
\hline
0.8&-18.33&-18.38&-18.29&0.49\%\\
\hline
    \end{tabular}
    \label{tab:casestudy_dr}
\end{table}


\noindent From Table~\ref{tab:casestudy_dr}, we observe that as the size of the ambiguity set ($\Delta$) increases, the objective value decreases, indicating reduced fairness and efficiency. A broader ambiguity set considers potential worst-case distributions, enabling the partial information-based SA-DR model \eqref{dro_model_sampled} to place more weight on fewer scenarios, which increases variance and further reduces fairness and efficiency. This also leads to a smaller gap between the upper and lower bounds. This is because we use the same maximum and minimum resource requirements across all replicates. When the ambiguity set is large, the model increasingly relies on the worst-case distribution, causing resource allocations to concentrate around these extreme values. As a result, the maximum (or minimum) resource levels are consistently triggered across replicates, leaving less gap between the upper and lower bounds.

\subsubsection{\textbf{Size of Uncertainty}}\label{subsec:size_uncertainty}
We now consider the resource allocation setting presented by \cite{Yao17}, using CloudSim \cite{Calheiros11} with box demand uncertainty, to analyze the impact of support size on allocations under settings of both full-information and partial-information settings (specifically, when only the mean and variance of the resource requirement distribution are available). CloudSim is a widely used tool for modeling and simulating cloud computing infrastructures and services without considering the low-level details of the cloud environments. We adopt the illustrative example used by \cite{Yao17}, which involves a cloud with four heterogeneous resource requirements on CPU, GPU, CPU memory, and GPU memory by 10 cloud users for running a game job. Jobs are infinitely divisible. The total available resources in the cloud include 80 CPUs, 900 GPUs, 160 GB CPU memory, and 1800 GB GPU memory. The parameters  are set to be $\beta=2$,  $\lambda=(1-\beta)/\beta=-1/2$ and $\theta=1$. This setting is used to illustrate the properties of the robust model without chance constraints, with the changing size of the support of the ambiguity set. By using box demand uncertainty, the uncertainty set for each user $j$ and resource $i$ is defined as
 $\Xi=\left\{r_{ij}:||r_{ij}-\hat{r}_{ij}||_{\infty} \leq \rho\right\},$
where $\rho$ is a measure of the size of the uncertainty, $r_{ij}$ is resource demand $i$ required by user $j$ and $\hat{r}_{ij}$ is the nominal resource demand $i$ required by user $j$. We consider $\rho \in [0.2,0.5,1.0,2.0]$.

First, we consider that we have the full information on the resource requirement distribution and apply the full information-based SAA model \eqref{SAA_model}. We generate 500 scenarios following the triangle distribution with the lower limit $\hat{r}_{ij}-\rho$, the upper limit $\hat{r}_{ij}+\rho$ and the mode $\hat{r}_{ij}$ and the configuration of the nominal resource $\hat{r}_{ij}$ is the same as \cite{Yao17}. 
Next, we consider that we can only have the partial information (i.e., bounds of the mean and variance) and apply the partial information-based SA-DR model \eqref{dro_model_sampled} with the ambiguity set  $\mathcal{P}^{|\Omega|}=\{p^\omega\mid\sum_{\omega=1}^{|\Omega|} p^\omega=1, p^\omega\geq 0 \; \forall \omega,0.5\hat{r}_{ij}\leq \sum_{\omega=1}^{|\Omega|}p^\omega r_{ij}^\omega\leq  1.5\hat{r}_{ij}, \sum_{\omega=1}^{|\Omega|}\left(p^\omega r_{ij}^\omega\right)^2-\left(\sum_{\omega=1}^{|\Omega|}p^\omega r_{ij}^\omega\right)^2\leq 1.5\hat{\sigma}_{ij}^2,\forall i,j\}$, where $\hat{\sigma}_{ij}$ is the variance of the resource requirements when the required resource follows the triangle distribution. We also apply the expected value-based deterministic model \eqref{fairness_stoch_obj_alter}, and the worst-case information-based robust model \eqref{fairness_stoch_obj_ro} to the same uncertainty set to consider the worst-case scenario. We evaluate each model's \textit{out-of-sample} performance by comparing (i) the SAA-objective (efficiency-fairness) value \eqref{SAA_model} and (ii) the leftover amounts of each of the four resources from that model's optimal allocations, both relative to the full-information SAA benchmark (Figures~\ref{fig:size} and \ref{fig:sizeleftovers}). 

\begin{figure}[h]
    \begin{center}
    \includegraphics[width=\linewidth]{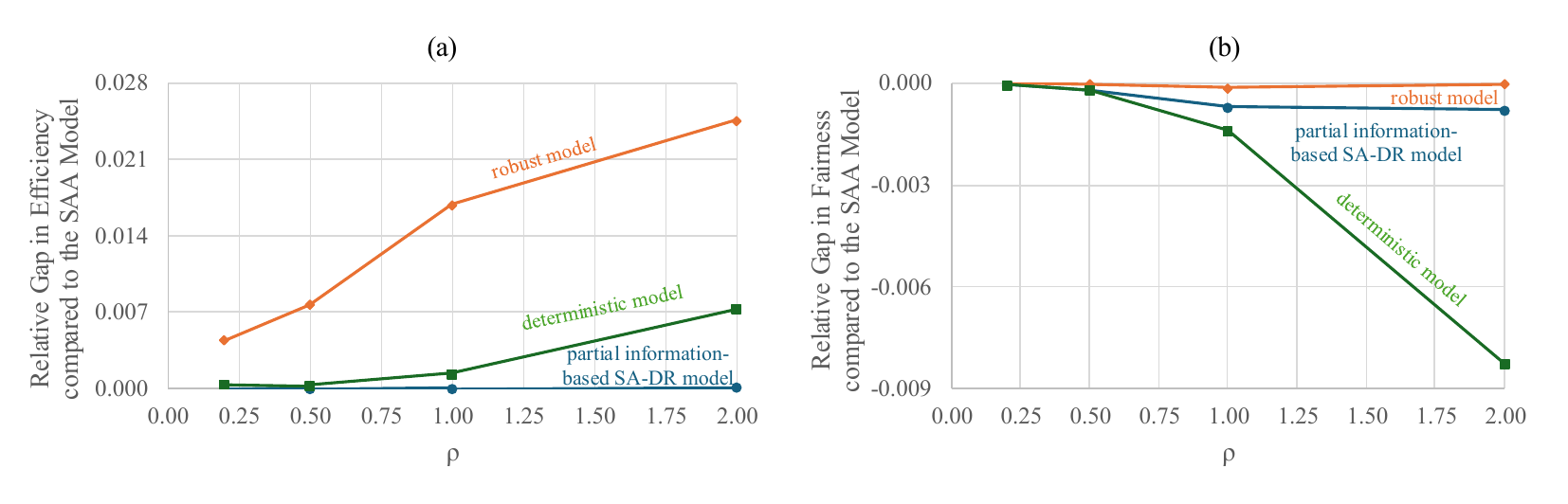}
    \caption{Change in performance of allocation models under full-, partial- worst-case information, and deterministic settings: (a) compares the relative gap in the optimal efficiency and the full-information SAA model, (b) compares the relative gap in the optimal fairness and the full-information SAA model} \label{fig:size}
    \end{center}
\end{figure} 

\begin{figure}[h]
    \begin{center}
    \includegraphics[width=\linewidth]{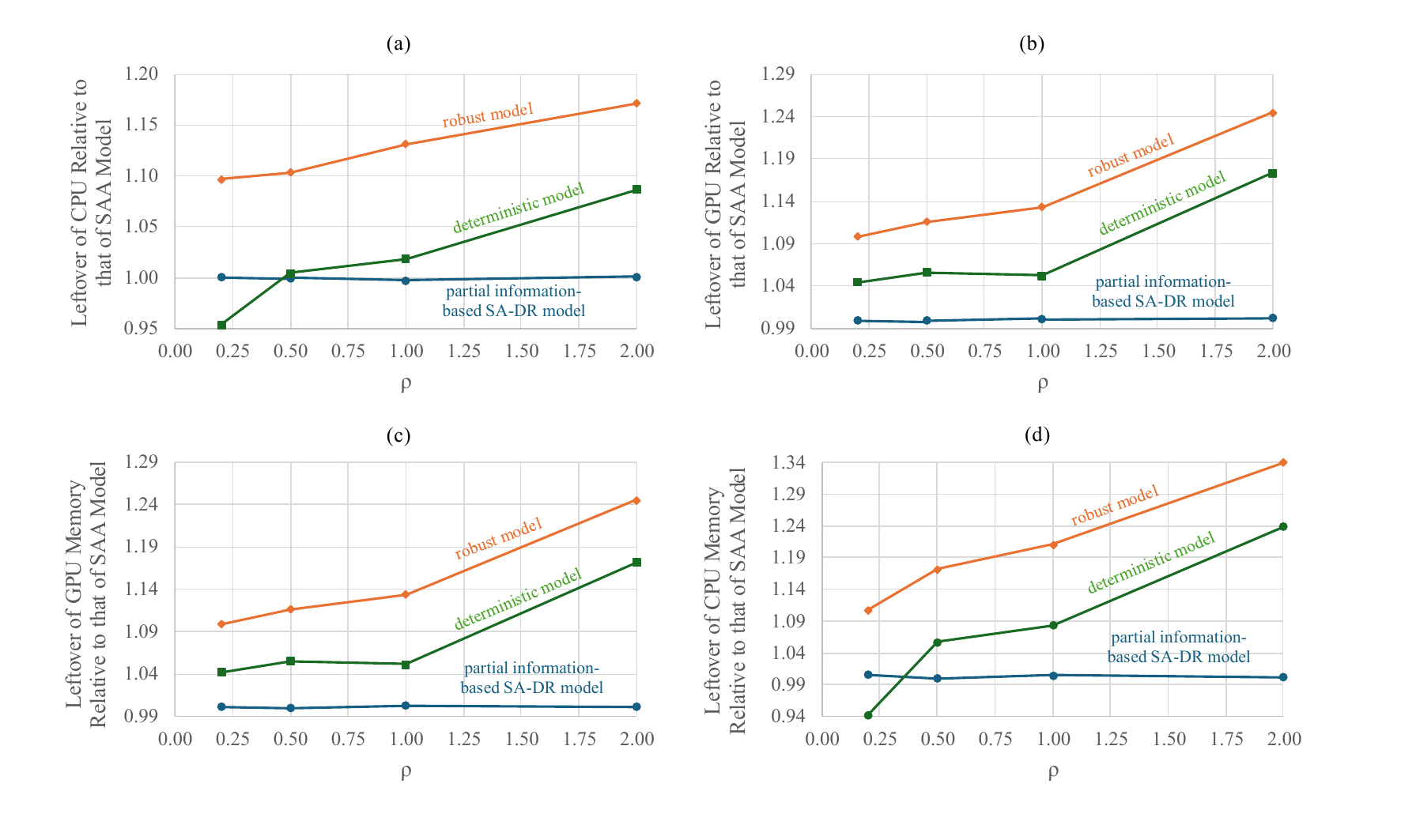}
    \caption{Leftover of each resource ((a) CPU, (b) GPU, (c) GPU Memory, (d) CPU Memory) under the optimal allocations from partial-information proposed model, worst-case information robust model, and expected value-based deterministic model relative to the full-information SAA model.} \label{fig:sizeleftovers}
    \end{center}
\end{figure} 

\noindent Figure~\ref{fig:size} compares the relative gap of fairness and efficiency among the allocations from the partial information-based SA-DR model \eqref{dro_model_sampled}, the worst-case information-based robust model \eqref{fairness_stoch_obj_ro}, and the expected value-based deterministic model \eqref{fairness_stoch_obj_alter} relative to that of the full-information SAA model \eqref{SAA_model}. When the size of the uncertainty set ($\rho$) increases, the robust and deterministic model shows an increase in the optimal efficiency relative to the SAA model. In contrast, the deterministic model shows a decrease in the optimal fairness relative to the other models. However, we observe that the mean- and variance-based SA-DR model consistently performs similarly to the full-information SAA model. 

Although the robust model shows higher efficiency and higher fairness than the mean- and variance-based SA-DR model, Figure~\ref{fig:sizeleftovers} illustrates that the optimal allocations from the robust model (and deterministic model) for each resource show higher leftovers relative to the SA-DR model and the SAA model. On the other hand, the leftovers from the partial information-based SA-DR model and the full-information SAA model are consistently similar to each other. Therefore, one can achieve similar efficiency and fairness by leveraging partial information on mean and variance instead of requiring full information on resource requirement distribution.

\section{Conclusions}\label{sec:conclusion}

In a cloud computing setting, we introduced a new stochastic fairness model for multi-resource allocations by extending the previously defined deterministic model FDS. To solve the cases where the distribution of the required resources is not known, we proposed a distributionally robust model with a moment uncertainty set and its sampled approximation. We focused on the worst-case scenario of the distribution to guarantee fair and efficient allocation for all potential distributions of the required resources. We introduced three key properties of the stochastic fairness functions: Stochastic Pareto-efficiency, Stochastic sharing incentive, and Stochastic envy-freeness. We derived the conditions for our proposed stochastic fairness model to satisfy these properties. We also provided a cutting surface algorithm to solve the proposed SA-DR model \eqref{dro_model_sampled}. 

To generate insights on the resource allocation decisions, we compared the performance of the expected value-based deterministic model \eqref{fairness_stoch_obj_alter}, the full information-based SAA model \eqref{SAA_model}, and the worst-case information-based robust model \eqref{fairness_stoch_obj_ro}, with our mean- and variance-based SA-DR model \eqref{dro_model_sampled} through numerical analysis. Using real-world cloud computing datasets, such as Azure and simulated CloudSim datasets, we found that the variance of the resource requirements leads to less equitable allocations under both full- and partial-information settings. We demonstrated that our mean- and variance-based SA-DR model \eqref{dro_model_sampled} performs closer to that of the full-information setting \eqref{SAA_model}. Moreover, the leftovers from the partial information-based SA-DR model and the full-information SAA model are consistently similar to each other. Thus, we illustrate that one can achieve similar efficiency and fairness by leveraging partial information on mean and variance instead of requiring full information on resource requirement distribution. 

For future studies, it is possible to reformulate the chance constraint in \eqref{dro_model} without using big-M. One such approach is the binary bilinear reformulation \cite{wang2021chance}, in which binary variables $z_i^\omega$ are introduced to write the chance constraint as \vspace{-2mm}
\begin{align*}
\sum_{j=1}^dr_{ij}^\omega x_jz_{i}^\omega&\leq c_iz_{i}^\omega\; \forall i=1,..,m, \forall \omega=1,...,|\Omega|, \\
\sum_{\omega=1}^{|\Omega|}p^\omega z_{i}^\omega &\geq \theta \; \forall i=1,..,m.
\vspace{-2mm}
\end{align*}
We observed that both formulations take the same amount of time to solve using Gurobi. It would be interesting to explore the incorporation of structured cuts (e.g., as in \cite{wang2021chance}) to develop a more efficient algorithm for solving \eqref{dro_model}.

Beyond reformulations of the static model, an interesting direction to explore would be the dynamic setting of multi-resource allocation. 
In practical cloud computing environments, users may arrive at different times with varying demands. This requires reallocating resources whenever a new user arrives and considering the unknown resource requirements of each user. A dynamic allocation model would provide valuable insights into how resource allocation strategies can be optimized in real time, thereby improving overall system performance.

\clearpage                  
\setcounter{page}{1}        
\renewcommand{\thepage}{A-\arabic{page}} 


\section*{Appendix for ``A Framework for Stochastic Fairness in Dominant Resource Allocation with Cloud Computing Applications''}

\section{Proofs for Properties of Fairness Function with Fairness Parameters}\label{property_proof}

\textbf{\textit{Proof of Proposition~\ref{prop:beta-monotone}}.}
Fix \(x\in\mathcal{X}\) and a realization \(\xi\) with \(S(\xi)>0\). Write \(t_j=\tilde t_j(x,\xi)\in[0,1]\) and
\[
A_\beta(t)\;:=\;\sum_{j=1}^{d} t_j^{\,1-\beta},\qquad \beta>1.
\]
For \(\beta>1\), the fairness component is
\[
F_\beta(x,\xi)\;=\;-\,[A_\beta(t)]^{1/\beta}.
\]
Define \(\phi(\beta):=\ln(-F_\beta(x,\xi))=(1/\beta)\ln A_\beta(t)\). It suffices to show \(\phi'(\beta)\ge 0\), since then \(-F_\beta\) is increasing and \(F_\beta\) is nonincreasing in \(\beta\).

Let
\[
w_j(\beta)\;:=\;\frac{t_j^{\,1-\beta}}{A_\beta(t)},\qquad j=1,\dots,d,
\]
so that \(w(\beta)\in\Delta_{d-1}\) and \(\frac{d}{d\beta}\ln A_\beta(t)=\sum_{j=1}^d w_j(\beta)\,(-\ln t_j)=:\mathbb{E}_{w(\beta)}[-\ln t_j]\).
Hence
\[
\phi'(\beta)\;=\;\frac{\beta\,\mathbb{E}_{w(\beta)}[-\ln t_j]-\ln A_\beta(t)}{\beta^2}.
\]
Using \(\ln t_j^{\,1-\beta}=\ln w_j(\beta)+\ln A_\beta(t)\) and averaging under \(w(\beta)\) gives
\[
\ln A_\beta(t)=(1-\beta)\,\mathbb{E}_{w(\beta)}[\ln t_j]-\sum_{j=1}^d w_j(\beta)\ln w_j(\beta).
\]
Therefore,
\[
\begin{aligned}
 \beta\,\mathbb{E}_{w(\beta)}[-\ln t_j]-\ln A_\beta(t)
&=\sum_{j=1}^d w_j(\beta)\ln\frac{w_j(\beta)}{t_j}\\
&=:D_{\mathrm{KL}}\!\big(w(\beta)\,\|\,t\big)\;\ge\;0,   
\end{aligned}
\]
and
\[
\phi'(\beta)\;=\;\frac{D_{\mathrm{KL}}(w(\beta)\,\|\,t)}{\beta^2}\;\ge\;0,
\]
with equality iff \(w(\beta)=t\), i.e., iff \(t\) is uniform on its positive support (or degenerate). Thus, \(F_\beta(x,\xi)\) is pointwise nonincreasing in $\beta$. Since \(S(\xi)^{\lambda}\) is \(\beta\)-independent, the same monotonicity holds for \(F(x,\xi)=F_\beta(x,\xi)S(\xi)^{\lambda}\).

Finally, by the boundedness assumption on \(F(x,\cdot)\) for each fixed \(x\), taking expectations preserves monotonicity, yielding the claimed statements for \(\mathbb{E}[F_\beta(x,\xi)]\) and \(\mathbb{E}[F(x,\xi)]\).
\hfill\(\square\)

\textbf{\textit{Proof of Proposition~\ref{prop:value-monotonicity-beta}.}}
Fix any allocation \(x\in\mathcal{X}\) and any \(\mathbb{P}\in\mathcal{P}(x)\). By Proposition~\ref{prop:beta-monotone} (fixed-\(x\) monotonicity), for \(\beta>1\) the mapping \(\beta\mapsto F_\beta(x,\xi)\) is almost surely nonincreasing, hence \(\beta\mapsto \mathbb{E}_{\mathbb{P}}[F_\beta(x,\xi)]\) is nonincreasing. Taking the infimum over \(\mathbb{P}\in\mathcal{P}(x)\) preserves nonincreasingness (an infimum of nonincreasing functions is nonincreasing), so for each fixed \(x\),
\[
q_{\mathrm{fair}}(x,\beta):=\min_{\mathbb{P}\in\mathcal{P}(x)}\mathbb{E}_{\mathbb{P}}[F_\beta(x,\xi)]
\;\text{is nonincreasing in \(\beta\) on }(1,\infty).
\]
Finally, the supremum (maximum) over \(x\in\mathcal{X}\) of a family of nonincreasing functions is also nonincreasing: for \(1<\beta_1<\beta_2\),
\[
v_{\mathrm{fair}}(\beta_2)
=\max_{x\in\mathcal{X}} q_{\mathrm{fair}}(x,\beta_2)
\ \le\ \max_{x\in\mathcal{X}} q_{\mathrm{fair}}(x,\beta_1)
=v_{\mathrm{fair}}(\beta_1).
\]

For the combined objective \(F(x,\xi)=F_\beta(x,\xi)\,S(\xi)^{\lambda}\), the factor \(S(\xi)^{\lambda}\) is \(\beta\)-independent, and the same argument applies verbatim: for each fixed \(x\) and \(\mathbb{P}\in\mathcal{P}(x)\), \(\beta\mapsto \mathbb{E}_{\mathbb{P}}[F(x,\xi)]\) is nonincreasing on \((1,\infty)\); thus
\[
q(x,\beta)\;:=\;\min_{\mathbb{P}\in\mathcal{P}(x)}\mathbb{E}_{\mathbb{P}}[F(x,\xi)]
\]
is nonincreasing in \(\beta\), and taking the maximum over \(x\) yields \(v(\beta_2)\le v(\beta_1)\). This proves the claim.

The proof for sampled-average corollary follows identical to Proposition~\ref{prop:value-monotonicity-beta}, using that for every fixed \(x\) and \(\mathbb{P}^{|\Omega|}\), each summand \(F_\beta(x,\xi^\omega)\) is nonincreasing in \(\beta\), hence the sampled expectations and the inner minima are nonincreasing; taking the maximum over \(x\) preserves nonincreasingness. \hfill\(\square\)

\textbf{\textit{Proof of Proposition~\ref{prop:eventual-monotone-infty}.}}
\emph{(a) Fairness component.}
For fixed \(x\) and \(\xi\), \(F_\beta(x,\xi)=-\big(\sum_j\tilde t_j(x,\xi)^{\,1-\beta}\big)^{1/\beta}\) is pointwise nonincreasing in \(\beta\) for \(\beta>1\) (Propostion~\ref{prop:beta-monotone}) and thus \(\mathbb{E}_{\mathbb{P}}[F_\beta(x,\xi)]\) is nonincreasing for any \(\mathbb{P}\). Taking the inner minimum over \(\mathcal{P}(x)\) and the outer maximum over \(x\in\mathcal{X}\) preserves nonincreasingness (as in Table~\ref{tab:stoch_fairness_cases}). By the boundedness assumptions in Section~\ref{subsec:stoch_model} and dominated convergence, we can pass the pointwise limit \(\beta\to\infty\) inside the expectation (Table~\ref{tab:stoch_fairness_cases}), yielding
\[
\lim_{\beta\to\infty}F_\beta(x,\xi)=-\max_j\frac{1}{\tilde t_j(x,\xi)}.
\]
Hence $v_{\mathrm{fair}}(\beta) \rightarrow \displaystyle -\,\mathbb{E}\!\left[\max_{j}\frac{1}{y_j(\xi)}\right]$.

\smallskip
\emph{(b) Combined objective with \(\lambda=(1-\beta)/\beta\).}
By the cancellation identity (Table~\ref{tab:stoch_fairness_cases}), for \(\beta>1\),
\[
F(x,\xi)=-\Big(\sum_{j=1}^{d} y_j(\xi)^{\,1-\beta}\Big)^{\!1/\beta},\qquad y_j(\xi)=\tilde\mu_j x_j.
\]
Fix \(x\) and \(\xi\), and define \(m:=m(x,\xi):=\min_j y_j(\xi)\) and the weights
\(
w_j(\beta):=y_j(\xi)^{\,1-\beta}\Big/\sum_\ell y_\ell(\xi)^{\,1-\beta}.
\)
Let
\[
\varphi_\beta(x,\xi)\;:=\;\ln\big(-F(x,\xi)\big)\;=\;\frac{1}{\beta}\,\ln\!\sum_{j=1}^d y_j(\xi)^{\,1-\beta}.
\]
A direct differentiation gives
\begin{align}\label{eq:phi-deriv}
\partial_\beta \varphi_\beta(x,\xi)&=\;\frac{1}{\beta^2}\Bigg[\sum_{j=1}^d w_j(\beta)\,\ln\frac{w_j(\beta)}{q_j(x,\xi)}\;-\;\ln S(\xi)\Bigg],\nonumber\\
q_j(x,\xi)&:=\frac{y_j(\xi)}{S(\xi)},\nonumber\\
 S(\xi)&:=\sum_{k=1}^d y_k(\xi).
\end{align}
To see this, note
\(
\frac{d}{d\beta}\ln\sum_j y_j^{1-\beta}=\sum_j w_j(\beta)\big(-\ln y_j\big),
\)
and use the identity
\(
\ln y_j^{1-\beta}=\ln w_j(\beta)+\ln\sum_\ell y_\ell^{1-\beta}
\)
to express \(\ln\sum_j y_j^{1-\beta}\) in terms of \(\sum_j w_j\ln w_j\) and \(\sum_j w_j\ln y_j\), then rearrange to obtain \eqref{eq:phi-deriv}.

As \(\beta\to\infty\), the weights \(w(\beta)\) concentrate on the set \(J_{\min}(x,\xi):=\{j: y_j(\xi)=m\}\).
Hence
\[
\begin{aligned}
 \lim_{\beta\to\infty}\partial_\beta \varphi_\beta(x,\xi)
=&\lim_{\beta\to\infty}\frac{1}{\beta^2}\Big[\mathrm{KL}\big(w(\beta)\,\|\,q(x,\xi)\big)-\ln S(\xi)\Big]=\frac{-\ln m(x,\xi)}{\beta^2}+o\!\left(\frac{1}{\beta^2}\right),   
\end{aligned}
\]
because 
$$\mathrm{KL}\big(w(\beta)\,\|\,q(x,\xi)\big)\rightarrow \ln q_{j^*}(x,\xi)=\ln\!\big(S(\xi)/m(x,\xi)\big)$$ for any \(j^*\in J_{\min}(x,\xi)\).
Therefore, for each fixed \((x,\xi)\) there exists \(\beta_{x,\xi}\) such that for all \(\beta\ge\beta_{x,\xi}\),
$$
\mathrm{sign}\big(\partial_\beta \varphi_\beta(x,\xi)\big)\;=\;\mathrm{sign}\big(-\ln m(x,\xi)\big),
$$
and thus
$$
\mathrm{sign}\big(\partial_\beta F(x,\xi)\big)\;=\;-\mathrm{sign}\big(\partial_\beta \varphi_\beta(x,\xi)\big).$$
(because \(F=-e^{\varphi_\beta}\)).
Equivalently, for large \(\beta\), \(F(x,\xi)\) is monotone in \(\beta\) with the direction determined by whether \(m(x,\xi)\) is below, above, or equal to \(1\):

$F(x,\xi)\ \rightarrow\ -\frac{1}{m(x,\xi)}.$
Now impose the assumption in Part~(b): there exist \(\delta>0\) and \(\beta_1>1\) such that \emph{every} optimizer \((x_\beta,\mathbb{P}_\beta)\) of \(v(\beta)\) satisfies either \(m(x_\beta,\xi)\le 1-\delta\) a.s. for all \(\beta\ge\beta_1\) (case \(<1\)), or \(m(x_\beta,\xi)\ge 1+\delta\) a.s. (case \(>1\)). This uniform nondegeneracy away from \(1\) ensures that the sign conclusion above holds almost surely with a \emph{common} threshold \(\beta_0\ge\beta_1\) for all optimal pairs \((x_\beta,\mathbb{P}_\beta)\).
Therefore, for \(\beta\ge\beta_0\), \(\beta\mapsto F(x_\beta,\xi)\) is a.s. monotone with a fixed direction under \(\mathbb{P}_\beta\), hence \(\beta\mapsto \mathbb{E}_{\mathbb{P}_\beta}[F(x_\beta,\xi)]\) is monotone with the same direction. As in Proposition~\ref{prop:value-monotonicity-beta}, taking the inner minimum (over \(\mathcal{P}(x)\)) and then the outer maximum (over \(\mathcal{X}\)) preserves monotonicity, which yields
\[
v(\beta) \rightarrow \displaystyle -\,\mathbb{E}\!\left[S(\xi)^{\lambda+1}\,\max_{j}\frac{1}{y_j(\xi)}\right].
\]
Dominated convergence (Section~\ref{subsec:stoch_model}) justifies the limiting value \(v(\infty)\) (the max-ratio benchmark). 

\hfill\(\square\)
\section{Convergence of Sample Based Approximation Model}\label{convergence_proof}
We now show that \eqref{dro_model_sampled} converges to \eqref{dro_model}. We first introduce results from the literature. 
\begin{corollary}\cite[Corollary 6.4.20]{sohrab2003basic}\label{bounded_firstDerivative}
A differentiable function $F:\mathbb{R}^d\rightarrow\mathbb{R}$ is Lipschitz on $\mathbb{R}^d$ if and only
if the first derivative of $F$ is bounded on $\mathbb{R}^d$.
\end{corollary}

\begin{definition}(\textit{Hausdorff Distances}) 
\cite[\textit{Definition 1}]{Lei24}\label{def:set_distance}
\textit{Let $A,B \subseteq \mathbb{R}^n$, $||\cdot||$ be the 2-norm for the vector and the Frobenius norm for the matrix,}  
\begin{align}\label{beta_distance}
\mathrm{d}(a,B):=\min_{b\in B} ||a-b||, \quad D(A, B):=\max_{a\in A} \mathrm{d}(a,B)\; ,
 \text{and} \quad \beta^{|\Omega|}:=D(\Xi,\Xi^{|\Omega|}).
\end{align}
\textit{Let $\mathbb{P}, \mathbb{Q}\in \mathcal{D}$, and $\mathcal{G}$ be a family of Lipschitz continuous functions with constant 1, i.e.,}
$  \mathcal{G} = \{g\mid
|g(\xi')-g(\xi'')|\leq ||\xi'-\xi''||,\; \forall \xi',\xi''\in \Xi\;\}. $
\textit{Then} 
\begin{equation}
\label{define_distanceP}
    \rho(\mathbb{Q},\mathbb{P}):=\sup _{g \in \mathcal{G}}\left|\mathbb{E}_{\mathbb{Q}}[g(\xi)]-\mathbb{E}_{\mathbb{P}}[g(\xi)]\right|,
\end{equation}
\textit{and the metric between the set ${\cal P}$ and $\mathbb{Q}$ is defined as:}
$
    \mathrm{d}(\mathbb{Q}, \mathcal{P})=\min_{\mathbb{P} \in \mathcal{P}}\rho(\mathbb{Q},\mathbb{P}).
$
\textit{Finally, the Hausdorff metric between $\mathcal{P}^{|\Omega|}$ and $\mathcal{P}$ is defined as} 
\begin{equation}
\mathrm{H}\left(\mathcal{P}^{|\Omega|}, \mathcal{P}\right):=\max \left\{\max_{\mathbb{P}^{|\Omega|}\in\mathcal{P}_{|\Omega|}} \mathrm{d}\left(\mathbb{P}^{|\Omega|}, \mathcal{P}\right), \max_{\mathbb{P}\in\mathcal{P}}  \mathrm{d}\left(\mathbb{P}, \mathcal{P}^{|\Omega|}\right)\right\}.
\end{equation}
\end{definition}

\begin{theorem}\cite[Theorem 12]{Liu19}\label{moment_set_converge}
Let $\mathcal{P}$ and $\mathcal{P}_{|\Omega|}$ be defined as $
\mathcal{P} =\{\mathbb{P}\mid\mathbb{E}_{\mathbb{P}}[\psi(\xi)]\in\mathcal{K} \}$,
$\mathcal{P}^{|\Omega|}=\{\mathbb{P}^{|\Omega|}\mid  \mathbb{P}^{|\Omega|}=\left(p^1,\dots,p^{|\Omega|}\right),\sum_{\omega=1}^{|\Omega|}p^\omega=1, p^\omega\geq 0\; \forall \omega, \mathbb{E}_{\mathbb{P}^{|\Omega|}}[\psi(\xi)]= \sum_{\omega=1}^{|\Omega|}p^\omega\psi(\xi^\omega)\in\mathcal{K}\},$
where $\psi(\cdot)$ is a mapping consisting of vectors and/or matrices and $\mathcal{K}$ is a closed convex cone in a matrix space.
If  $\mathcal{P}$ satisfies the Slater condition, i.e., there exists $\mathbb{P}_0 \in\mathcal{P}$ and a constant $\alpha>0$ such that
$\mathbb{E}_{\mathbb{P}_0}[\psi(\xi)] +\alpha\mathcal{B} \subset \mathcal{K},$ 
where $\mathcal{B}$ is the unit ball in the space of $\mathcal{K}$ and $+$ is the Minkowski sum,  
and the function $\psi(\xi)$ is Lipshitz continuous with Lipschitz constant $\kappa^\psi$,
then for sufficiently large ${|\Omega|}$, 
$\mathrm{H}\left(\mathcal{P}^{|\Omega|}, \mathcal{P}\right) \leq \left(1+\frac{2\kappa^\psi||\textbf{1}||M_\Xi}{\alpha}\right)\beta^{|\Omega|}$, where $\bf{1}$ is a matrix that has the same size as $\psi(\cdot)$ with each component being 1.
\end{theorem}
\begin{theorem}\cite[Theorem 1]{henrion2004perturbation}\label{perturbation_lipschitz}
For a given probability $\mathbb{P}$,  multifunction $H(x)$ and the objective function $g(x)$,  the general chance constraint program is $\min\{g(x)|x\in\mathcal{X},\text{Prob}_\mathbb{P}\left(H(x) \right)\geq \theta\}$. Let $\Psi$ be the constraint set mapping, i.e. $\Psi=\{x\in\mathcal{X}|\text{Prob}_\mathbb{P}\left(H(x) \right)\geq \theta\}$. If $\Psi^{-1}$ is metrically regular at all feasible solutions $x$ for the given $\theta$, $\mathcal{X}$ is compact and $g(x)$ is Lipschitzian in $x$, then the optimal value $\min\{g(x)|x\in\Psi\}$ is locally Lipschitz continuous in $\theta$ with Lipschitz constant $\kappa^\theta$. 
\end{theorem}

The following is a restatement of Corollary 1 in \cite{Lei24}, in which all assumptions are explicitly stated. 
\begin{theorem}\cite[Corollary 1]{Lei24}\label{coro:beta_converge_dist}
Let $\xi^1, \xi^2,\dots,\xi^{|\Omega|}$ be uniformly distributed on the support $\Xi^{|\Omega|}$. Suppose that the following conditions hold: \textbf{C1}: The functions $F(x, \xi)$, $G(x,\xi)$ are measurable and Lipschitz continuous in $\xi$ with Lipschitz constants $\kappa^F$, $\kappa^G$, respectively, \textbf{C2}: The optimal value $v_\mathbb{P}(\theta):=\min_{x\in\mathcal{X}}\{\mathbb{E}_\mathbb{P}[F(x,\xi)]|\text{Prob}_\mathbb{P}\left(G(x,\xi)\leq 0 \right)\geq \theta\}$ is locally Lipschitz continuous in $\theta$ with Lipschitz constant $\kappa^\theta$, and \textbf{C3}: The continuous distributions $\mathbb{P}\in\mathcal{P}$ have a bounded probability density $f_\mathbb{P}(\cdot)$ with the bound $C^\mathcal{P}$ and $\mathcal{P}$ satisfies Theorem~\ref{moment_set_converge}. 
Let $(\hat{x},\hat{\mathbb{P}}_{|\Omega|})$ be an optimal solution and $\hat{v}^{|\Omega|}$ be the corresponding optimal value of the following model: 
\begin{equation}
\begin{aligned}
\hat{v}^{|\Omega|}:=\min_{x\in\mathcal{X}} \max_{\mathbb{P}^{|\Omega|}\in \mathcal{P}^{|\Omega|}}\;\;& \mathbb{E}_{\mathbb{P}^{|\Omega|} }[F(x,\xi)]\\
 \text{s.t.}\;\; &\text{Prob}_{\mathbb{P}^{|\Omega|}}\left(G(x,\xi)\leq  0 \right)\geq \theta.
\end{aligned}
\end{equation}
Let $(x^*,\mathbb{P}^*)$ be an optimal solution and $v^*$ be the corresponding optimal value of the following model:  
\begin{equation}
\begin{aligned}
v^*:=\min_{x\in\mathcal{X}}  \max_{\mathbb{P}\in \mathcal{P}}\;\; &\mathbb{E}_\mathbb{P}[F(x,\xi)]\quad \\
\text{s.t.}\;\;&\text{Prob}_\mathbb{P}\left(G(x,\xi)\leq 0 \right)\geq \theta.
\end{aligned}
\end{equation}
For any $\varepsilon>0$, a sufficiently large $|\Omega|_0^\varepsilon$ such that when $|\Omega|>|\Omega|_0^\varepsilon$, 
$$|v^*-\hat{v}^{|\Omega|}|<\frac{\kappa^FC^H}{2}\left(\frac{W}{|\Omega|}\right)^{1/d}+\kappa^\theta\sqrt{\kappa^GC^\mathcal{P}C^H\left(\frac{W}{|\Omega|}\right)^{1/d}},$$
with probability 1, where $W:=\log |\Omega|+(d-1+\varepsilon)\log \log |\Omega|$.
\end{theorem}

Using Corollary~\ref{bounded_firstDerivative} and Theorems~\ref{moment_set_converge}, \ref{perturbation_lipschitz}, and \ref{coro:beta_converge_dist}, we now demonstrate that \eqref{dro_model_sampled} converges to \eqref{dro_model}, as stated in Theorem~\ref{thm:beta_converge_dist}.

\begin{theorem}
\label{thm:beta_converge_dist}
Let $\xi^1, \xi^2,\dots,\xi^{|\Omega|}$ be uniformly distributed. Let $(\hat{x},\hat{\mathbb{P}}_{|\Omega|})$ be an optimal solution of \eqref{dro_model_sampled} and $\hat{v}^{|\Omega|}$ be the corresponding optimal value. Let $(x^*,\mathbb{P}^*)$ be an optimal solution of \eqref{dro_model} and $v^*$ be the corresponding optimal value. There exist positive constants $\kappa^F$, $\kappa^G$, $\kappa^\theta$, $C^H$ and $C^P$. Then for any $\varepsilon>0$, there exists a sufficiently large $|\Omega|_0^\varepsilon$ such that when $|\Omega|>|\Omega|_0^\varepsilon$, 
$$|v^*-\hat{v}^{|\Omega|}|<\frac{\kappa^FC^H}{2}\left(\frac{W}{|\Omega|}\right)^{1/d}+\kappa^\theta\sqrt{\kappa^GC^\mathcal{P}C^H\left(\frac{W}{|\Omega|}\right)^{1/d}},$$
with probability 1, where $W:=\log |\Omega|+(d-1+\varepsilon)\log \log |\Omega|$.
\end{theorem}
 \textit{Proof.}
We first show that $F(x,\xi)$ is Lipschitz continuous in $\xi$ (i.e., Condition \textbf{C1} of Theorem~\ref{coro:beta_converge_dist}). 
We have the 2-norm of the first derivative of $F(x,\xi)$ on $\xi$ ($||(\partial F(x,\xi)/\partial \xi)||$) as:
\begin{equation}
    \begin{aligned}
\Bigg(\sum_{k=1}^d\Big(\tilde{\mu}_k^{-\beta}x_k^{1-\beta}\left(\sum_{j=1}^d(\tilde{\mu}_jx_j)^{1-\beta}\right)^\frac{1-\beta}{\beta}\left(\sum_{j=1}^{d} \tilde{\mu}_j x_{j}\right)^{\lambda-\frac{1-\beta}{\beta}}
+(\lambda-\frac{1-\beta}{\beta})x_k(\sum_{j=1}^d\tilde{\mu}_jx_j)^{\lambda-\frac{1}{\beta}}\left(\sum_{j=1}^d(\tilde{\mu}_jx_j)^{1-\beta}\right)^\frac{1}{\beta}\Big)^2\Bigg)^{1/2}. 
\nonumber
    \end{aligned}
\end{equation}
Since $\mu$ is bounded,  $||(\partial F(x,\xi)/\partial \xi)||$ is bounded for any given $x\in\mathcal{X}$. By Corollary~\ref{bounded_firstDerivative}, if the first derivative of $F(x,\xi)$ is bounded, then $F(x,\xi)$ is Lipschitz continuous in $\xi$ and there exists the Lipschitz constant $\kappa^F$.
Then for $G(x,\xi)=\sum_{j=1}^{d} \tilde{r}_{ij} x_{j}-c_i$, $\left|\left|\left(\frac{\partial G(x,\xi)}{\partial \xi}\right)\right|\right|
=\sqrt{\sum_{j=1}^d x_j^2}<\infty. $

Since the first derivative of $G(x,\xi)$ is bounded,  $G(x,\xi)$ is Lipschitz continuous in $\xi$, and the Lipschitz constant $\kappa^G$ exists.

Next, we establish Condition \textbf{C2} of Theorem~\ref{coro:beta_converge_dist}. For any feasible solution $x$ for a given $\theta$, $\exists h\in\mathbb{R}^d$ such that $\theta-\text{Prob}_{\mathbb{P}^*}\left(\sum_{j=1}^{d} \tilde{r}_{ij} x_{j} \leq c_{i}\right)-\Big(\partial\Big(\text{Prob}_{\mathbb{P}^*}\left(\sum_{j=1}^{d} \tilde{r}_{ij} x_{j} \leq c_{i}\right)\Big)/\partial x\Big)h<0$ and $g(x)+(\partial g(x)/\partial x)h<0$ (Assumption~\ref{feasible_assumption}), according to \cite[Proposition 3.3]{bonnans1998optimization}, $\Psi^{-1}$ is metrically regular at all feasible solutions $x$ for the given $\theta$.  We have $||(\partial F(x,\xi)/\partial x)||$ as
\begin{equation}
\begin{aligned}
\Bigg(\sum_{k=1}^d\Big(\tilde{\mu}_k^{1-\beta}x_k^{-\beta}\left(\sum_{j=1}^d(\tilde{\mu}_jx_j)^{1-\beta}\right)^\frac{1-\beta}{\beta}\left(\sum_{j=1}^{d} \tilde{\mu}_j x_{j}\right)^{\lambda-\frac{1-\beta}{\beta}}
+(\lambda-\frac{1-\beta}{\beta})\tilde{\mu}_k(\sum_{j=1}^d\tilde{\mu}_jx_j)^{\lambda-\frac{1}{\beta}}\left(\sum_{j=1}^d(\tilde{\mu}_jx_j)^{1-\beta}\right)^\frac{1}{\beta}\Big)^2\Bigg)^{1/2},  
\end{aligned} 
\nonumber
\end{equation}
which is bounded for any $\xi\in\Xi$ since $\mathcal{X}$ is bounded. Then $||(\partial \mathbb{E}_\mathbb{P}[F(x,\xi)]/\partial x||)$ is bounded. By Corollary~\ref{bounded_firstDerivative},  $\mathbb{E}_\mathbb{P}[F(x,\xi)]$ is Lipschitz continuous in $x$. Then by Theorem~\ref{perturbation_lipschitz}, the optimal value $v_\mathbb{P}(\theta)$ is locally Lipschitz continuous in $\theta$ with Lipschitz constant $\kappa^\theta$.

For the ambiguity set \eqref{ambiguity_dro} and \eqref{ambiguity_dro_sampled}, we have {\small $\psi(\xi)=\begin{pmatrix}
    -\hat{u}_{ij}+\underline{u}_{ij}\\
    \hat{u}_{ij}-\bar{u}_{ij}\\
    (r_{ij}-\hat{u}_{ij})^2-(\mathbb{E}[r_{ij}]-\hat{u}_{ij})^2-\bar{\sigma}^2_{ij}
  \end{pmatrix}$}, then 
 $\alpha=\min\{\hat{u}_{ij}-\underline{u}_{ij},\bar{u}_{ij}-\hat{u}_{ij},\bar{\sigma}^2_{ij}-\hat{\sigma}^2_{ij}\}$, $\kappa^\psi=\sqrt{(2+16M_\Xi^2)md}$ and $\|\textbf{1}\|=\sqrt{3md}$. Using Theorem~\ref{moment_set_converge}, we have $\mathrm{H}\left(\mathcal{P}^{|\Omega|}, \mathcal{P}\right) \leq \left(1+\frac{2\sqrt{(2+16M_\Xi^2)md}M_\Xi\sqrt{3md}}{\alpha}\right)\beta^{|\Omega|}$. Thus, there exists a positive constant $C^H=1+\frac{2\sqrt{(2+16M_\Xi^2)md}\delta_\Xi\sqrt{3md}}{\alpha}$.  

Also, the existence of $C^\mathcal{P}$ is part of Assumption~\ref{feasible_assumption} (i.e., Condition \textbf{C3} of Theorem~\ref{coro:beta_converge_dist}). Finally, using Theorem~\ref{coro:beta_converge_dist}, we complete the proof.
$~\Box$

\section{Stochastic Fairness Properties}\label{proof:properties}
\subsection{Stochastic Pareto-Efficiency}\label{subsec:pareto}
We begin with the definition and parameter conditions for deterministic Pareto-efficiency. \begin{definition}\label{deterministic_ParetoEfficiency}
\cite[\textit{Definition 2}]{Joe13} \textit{A function $f$ is Pareto-efficient if, whenever $x$ Pareto-dominates $y$ (i.e., $x_i \geq y_i$ for each index $i$ and $x_j > y_j$ for some $j$), $f(x) > f(y)$.}
\end{definition}
\begin{proposition}\label{deterministic_ParetoEfficiency_prop}
\cite[Proposition 3]{Joe13} The fairness function FDS (\ref{fairness_deterministic}) is Pareto-efficient if and only if $|\lambda|\geq |(1-\beta)/\beta|$.
\end{proposition}

We now provide the proof of Proposition~\ref{prop:stochastic_pareto}.

\textbf{\textit{Proof of Proposition~\ref{prop:stochastic_pareto}}.}
Assume  $|\lambda|\geq |(1-\beta)/\beta|$, and that $x$ Pareto-dominates $y$. Then from Definition~\ref{deterministic_ParetoEfficiency} and Proposition~\ref{deterministic_ParetoEfficiency_prop},  we have 
$$F(x,\xi^\omega)>F(y,\xi^\omega)\quad \forall \omega.$$

Therefore, for any $p^\omega\in \mathcal{P}^{|\Omega|}$, we have 
$$\sum_{\omega=1}^{|\Omega|}p^\omega F(x,\xi^\omega)>\sum_{\omega=1}^{|\Omega|}p^\omega F(y,\xi^\omega),$$
or equivalently 
$\hat{f}^{|\Omega|}(x)>\hat{f}^{|\Omega|}(y)$, implying stochastic Pareto efficiency (Definition~\ref{def2}).
$~\Box$

\subsection{Stochastic Sharing Incentive}\label{subsec:sharing_incentive}
We begin with the definition and parameter conditions for deterministic sharing incentive. 
\begin{definition}
\cite[\textit{Definition 3}]{Joe13} \textit{Sharing incentive is the property that no user's dominant share is less than $1/d$, i.e., $\mu_j=\max_j\frac{r_{ij}}{c_i}$ and $\mu_jx_j\geq 1/{d}$.}
\end{definition}
\begin{proposition}
\cite[Proposition 4]{Joe13} 
Sharing incentive is satisfied by the FDS-optimal allocation when $\lambda=(1-\beta)/\beta$ and $\beta>1$. For $0 \leq \beta \leq 1$ and $\lambda=(1-\beta)/\beta$, sharing incentive may not be satisfied.
\end{proposition}
As a special case, when $\lambda=(1-\beta)/\beta$ and $\beta>1$, Corollary~\ref{corollary_convexity} shows that the objective function of \eqref{dro_model_sampled} is concave. 

\begin{corollary}\label{corollary_convexity}
Suppose $\beta>1$ and $\lambda=(1-\beta)/\beta$, for any given $\hat{p}$ and $\hat{z}$, 
the objective function of \eqref{dro_model_sampled} is concave and the model has a unique solution $x^*$. 
\end{corollary}
\begin{proof}
According to \eqref{alpha_fairness}, when $\lambda=(1-\beta)/\beta$ and $\beta\neq 1$, maximizing the objective function of \eqref{dro_model_sampled} is equivalent to optimizing the ($\alpha,p$)-fairness function, which means it can be simplified as  \vspace{-2mm}
\begin{equation}\label{convexity_obj}
\begin{aligned}
\max_{x\in\mathcal{X}} \quad &\sum_{\omega=1}^{|\Omega|} \sum_{j=1}^{d} \hat{p}^\omega\frac{\left(\mu_{j}^\omega x_{j}\right)^{1-\beta}}{1-\beta}.
\end{aligned}
\end{equation}
As $\beta>1$, $\hat{p}^\omega$ and $\mu_j^\omega$ are constants. 
For $\delta\in[0,1]$ and $\forall j,\omega$, we have 
$$\delta\left(\mu_{j}^\omega x_{j}\right)^{1-\beta}+(1-\delta)\left(\mu_{j}^\omega y_{j}\right)^{1-\beta}\geq \left(\mu_{j}^\omega (\delta x_{j}+(1-\delta)y_j)\right)^{1-\beta}.$$
With $1-\beta<0$, we have 
$$
\begin{aligned}
 \delta\sum_{\omega=1}^{|\Omega|} \sum_{j=1}^{d} \hat{p}^\omega\frac{\left(\mu_{j}^\omega x_{j}\right)^{1-\beta}}{1-\beta}+(1-\delta)\sum_{\omega=1}^{|\Omega|} \sum_{j=1}^{d} \hat{p}^\omega\frac{\left(\mu_{j}^\omega y_{j}\right)^{1-\beta}}{1-\beta}
 \leq \sum_{\omega=1}^{|\Omega|} \sum_{j=1}^{d} \hat{p}^\omega\frac{\left(\mu_{j}^\omega (\delta x_{j}+(1-\delta)y_j)\right)^{1-\beta}}{1-\beta}.   
\end{aligned}
$$

Therefore, the objective function \eqref{convexity_obj} is concave in $x$. 
As a maximization problem, it has an optimal solution, $x^*$.
\end{proof}
Based on Corollary~\ref{corollary_convexity}, we provide the proof of Proposition~\ref{stochastic_sharingincentive_prop}.

\textbf{\textit{Proof of Proposition~\ref{stochastic_sharingincentive_prop}}.}
Assume $\hat{z}$ and $\hat{p}^\omega$ are the optimal solutions of \eqref{dro_model_sampled}. We define a set for $i$: 
$\hat{\Omega}(i)=\{\omega|\hat{z}_{i}^\omega=1\}$.
As defined before, $\eta_{ij}^\omega=\frac{r_{ij}^\omega}{c_i}$. Based on Corollary~\ref{corollary_convexity},  \eqref{dro_model_sampled} has the same solution $x$ as the following model under the given condition:
\begin{equation}
    \begin{aligned}
\max_{x\in\mathcal{X}} \quad &
\sum_{\omega=1}^{|\Omega|} \sum_{j=1}^{d} \hat{p}^\omega\frac{\left(\mu_{j}^\omega x_{j}\right)^{1-\beta}}{1-\beta}\\
 \text{s.t.} \quad & \sum_{j=1}^{d} \eta_{i j}^\omega x_{j} \leq 1 \quad  \forall \omega\in \hat{\Omega}(i),\forall i.
\label{sharing_constr_DRO}
    \end{aligned}
\end{equation}

Let $\tau^\omega$ be Lagrange multipliers for \eqref{sharing_constr_DRO} and consider the Lagrangian function of  \eqref{sharing_constr_DRO}: 
\begin{equation}\label{Lagranian_DRO}
    L(x,\tau)=\sum_{\omega=1}^{|\Omega|} \sum_{j=1}^{d} \hat{p}^\omega\frac{\left(\mu_{j}^\omega x_{j}\right)^{1-\beta}}{1-\beta}-\sum_{\omega\in \hat{\Omega}}\sum_{i=1}^{m} \tau_{i}^\omega\left(\sum_{j=1}^{d} \eta_{i j}^\omega x_{j}-1\right).
\end{equation}

\noindent Based on \cite[\textit{Theorem 4.7 and Theorem 4.8}]{Ruszczynski11}, if $\hat{x}$ satisfies the first-order optimality conditions with Lagrange multiplier $\hat{\tau}$, $(\hat{x},\hat{\tau})$ is a saddle point of the Lagrangian function \eqref{Lagranian_DRO} and this saddle point is a solution of \eqref{sharing_constr_DRO}. Then, according to \cite[\textit{Theorem 4.10}]{Ruszczynski11}, when $(\hat{x},\hat{\tau})$ is a saddle point of the Lagrangian \eqref{Lagranian_DRO}, the necessary KKT conditions hold, which means the following inequalities hold:
\begin{align}
 &\frac{\partial L(x,\tau)}{\partial x}=\sum_{\omega=1}^{|\Omega|} \hat{p}^\omega\left(\mu_{j}^\omega\right)^{1-\beta} \left(x_{j}\right)^{-\beta}-\sum_{\omega\in\hat{\Omega}}\sum_{i=1}^{m} \tau_{i}^\omega \eta_{i j}^{\omega}=0\; \forall j,\label{KKT_lag_DRO}\\
&\tau_{i}^\omega\left(\sum_{j=1}^{d} \eta_{i j}^\omega x_{j}-1\right)=0\quad \forall \omega\in \hat{\Omega}(i),\forall i,\label{KKT_slack_DRO}\\
& \sum_{j=1}^{d} \eta_{i j}^\omega x_{j} \leq 1 \quad \forall \omega\in \hat{\Omega}(i),\forall i.
\end{align}
Here \eqref{KKT_slack_DRO} has the equivalent form $\sum_{\omega\in \hat{\Omega}}\sum_{i=1}^{m} \tau_{i}^\omega\left(\sum_{j=1}^{d} \eta_{i j}^\omega x_{j}-1\right)=0$.
For $\omega\not\in\hat{\Omega}(i)$, let $\tau_{i}^\omega=0$, then we can rewrite \eqref{KKT_lag_DRO} as: 
\begin{equation}\label{sharing_lagrangian_DRO}
    \sum_{\omega=1}^{|\Omega|} \hat{p}^\omega\left(\mu_{j}^\omega\right)^{1-\beta} \left(x_{j}\right)^{-\beta}=\sum_{i=1}^{m} \sum_{\omega\in\hat{\Omega}(i)}\tau_{i}^\omega \eta_{i j}^\omega=\sum_{\omega=1}^{|\Omega|}\sum_{i=1}^{m} \tau_{i}^\omega \eta_{i j}^\omega \; \forall j.
\end{equation}
\noindent By dividing $\mu_{j}^\omega$ on both sides and as $\frac{\eta_{i j}^\omega}{\mu_{j}^\omega}\leq 1$ (because $\mu^\omega_j:=\max_i\{\eta^\omega_{ij}\}$), we get the inequality:
\begin{equation}\label{sharing_inequality_DRO}
\sum_{\omega=1}^{|\Omega|}\hat{p}^\omega\left(\mu_{j}^\omega\right)^{-\beta} \left(x_{j}\right)^{-\beta}=\sum_{\omega=1}^{|\Omega|}\sum_{i=1}^{m} \tau_{i}^\omega \frac{\eta_{i j}^\omega}{\mu_{j}^\omega} \leq \sum_{\omega=1}^{|\Omega|}\sum_{i=1}^{m} \tau_{i}^\omega \quad \forall j.
\end{equation}

\noindent Then by \eqref{KKT_slack_DRO}, we have:
\begin{equation}\label{sharing_slackness_DRO}
\sum_{\omega\in\hat{\Omega}}\sum_{i=1}^{m} \sum_{j=1}^{d} \tau_{i}^\omega \eta_{i j}^\omega x_{j}=\sum_{\omega\in\hat{\Omega}}\sum_{i=1}^{m} \tau_{i}^\omega=\sum_{\omega=1}^{|\Omega|}\sum_{i=1}^{m} \tau_{i}^\omega.
\end{equation}

\noindent Combining \eqref{sharing_lagrangian_DRO} and \eqref{sharing_slackness_DRO}, we have: 
\begin{equation}\label{sharing_combined_DRO}
   \sum_{\omega=1}^{|\Omega|}\sum_{j=1}^{d} \hat{p}^\omega\left(\mu_{j}^\omega\right)^{1-\beta} \left(x_{j}\right)^{1-\beta}=\sum_{\omega=1}^{|\Omega|}\sum_{i=1}^{m} \tau_{i}^\omega.
\end{equation}

\noindent Therefore, using  \eqref{sharing_inequality_DRO} and  \eqref{sharing_combined_DRO}, we have:
\begin{equation}\label{muxb_inequality}
\sum_{\omega=1}^{|\Omega|} \hat{p}^\omega\left(\mu_{j}^\omega\right)^{-\beta} \left(x_{j}\right)^{-\beta}\leq \sum_{\omega=1}^{|\Omega|}\sum_{j=1}^{d} \hat{p}^\omega\left(\mu_{j}^\omega\right)^{1-\beta} \left(x_{j}\right)^{1-\beta}\quad \forall j.
\end{equation}
\noindent By summing over $j$, we have 
\begin{equation}\label{sum_j}
\sum_{\omega=1}^{|\Omega|} \sum_{j=1}^d\hat{p}^\omega\left(\mu_{j}^\omega x_{j}\right)^{-\beta}\leq d\sum_{\omega=1}^{|\Omega|}\sum_{j=1}^{d} \hat{p}^\omega\left(\mu_{j}^\omega x_{j}\right)^{1-\beta}.
\end{equation}

\noindent Let $X^\omega=\mu_{j}^\omega x_j$, then by Jensen's inequality, $\mathbb{E}[(X^\omega)^{-\beta}]\geq (\mathbb{E}[X^\omega])^{-\beta}$ as $\beta>1$, \eqref{muxb_inequality} can be written as 
\begin{equation}
\begin{aligned}
  \left(\sum_{\omega=1}^{|\Omega|} \hat{p}^\omega\mu_{j}^\omega x_{j}\right)^{-\beta}&\leq \sum_{\omega=1}^{|\Omega|} \hat{p}^\omega\left(\mu_{j}^\omega x_{j}\right)^{-\beta}\leq \sum_{\omega=1}^{|\Omega|}\sum_{j=1}^{d} \hat{p}^\omega\left(\mu_{j}^\omega x_{j}\right)^{1-\beta}\;\forall j.  
\end{aligned}
\end{equation}
which is equivalent to 
\begin{equation}\label{muxb_inequality_jensen}
\sum_{\omega=1}^{|\Omega|} \hat{p}^\omega(\mu_{j}^\omega x_{j})\geq \left(\sum_{\omega=1}^{|\Omega|}\sum_{j=1}^{d} \hat{p}^\omega\left(\mu_{j}^\omega x_{j}\right)^{1-\beta}\right)^{-\frac{1}{\beta}}\quad \forall j.
\end{equation}

\noindent Then by Hölder's inequality, for all $\omega=1,\dots, |\Omega|$, 
\begin{equation}\label{mux_1_b}
\begin{aligned}
 \sum_{j=1}^{d} \left(\mu_{j}^\omega x_j\right)^{1-\beta} &=\sum_{j=1}^{d} 1\cdot\left(\mu_{j}^\omega x_j\right)^{1-\beta} \leq \left(\sum_{j=1}^{d} 1^\beta\right)^\frac{1}{\beta}\left(\sum_{j=1}^{d} \left(\left(\mu_{j}^\omega x_j\right)^{1-\beta}\right)^{\frac{\beta}{\beta-1}}\right)^\frac{\beta-1}{\beta}\\
 &= d^\frac{1}{\beta}\left(\sum_{j=1}^{d} \left(\mu_{j}^\omega x_j\right)^{-\beta}\right)^\frac{\beta-1}{\beta}.
\end{aligned}
\end{equation}

\noindent Let $Y^\omega=\left(\sum_{j=1}^{d} \left(\mu_{j}^\omega x_j\right)^{-\beta}\right)$, then by Jensen's inequality, $\mathbb{E}[(Y^\omega)^\frac{\beta-1}{\beta}]\leq (\mathbb{E}[Y^\omega])^\frac{\beta-1}{\beta}$ as $0<\frac{\beta-1}{\beta}<1$, using \eqref{sum_j} and \eqref{mux_1_b}, we have 
\begin{equation}\label{d_mux}
    \begin{aligned}
  \sum_{\omega=1}^{|\Omega|} \sum_{j=1}^{d} \hat{p}^\omega \left(\mu_{j}^\omega x_j\right)^{1-\beta}& \leq d^\frac{1}{\beta}\sum_{\omega=1}^{|\Omega|} \hat{p}^\omega  \left(\sum_{j=1}^{d} \left(\mu_{j}^\omega x_j\right)^{-\beta}\right)^\frac{\beta-1}{\beta}\\
  & \leq   d^\frac{1}{\beta}\left(\sum_{\omega=1}^{|\Omega|} \sum_{j=1}^{d}\hat{p}^\omega   \left(\mu_{j}^\omega x_j\right)^{-\beta}\right)^\frac{\beta-1}{\beta} \\
  & \leq  d^\frac{1}{\beta}\left(d\sum_{\omega=1}^{|\Omega|}\sum_{j=1}^{d} \hat{p}^\omega\left(\mu_{j}^\omega x_{j}\right)^{1-\beta}\right)^\frac{\beta-1}{\beta}.
    \end{aligned}
\end{equation}
By reorganizing \eqref{d_mux}, we have 
\begin{equation}
    \left(\sum_{\omega=1}^{|\Omega|}\sum_{j=1}^{d} \hat{p}^\omega\left(\mu_{j}^\omega x_{j}\right)^{1-\beta}\right)^{-\frac{1}{\beta}}\geq \frac{1}{d}.
\end{equation}
Using \eqref{muxb_inequality_jensen}, we have 
\begin{equation}
    \sum_{\omega=1}^{|\Omega|} \hat{p}^\omega(\mu_{j}^\omega x_{j})\geq \frac{1}{d}\quad \forall j. 
\end{equation}
$~\Box$

\subsection{Stochastic Envy-Freeness}\label{subsec:envy_free}
We begin with the definition and parameter conditions of deterministic envy-freeness. 
\begin{definition}
\textit{\cite[Definition 4]{Joe13} User $j$ envies user $k$ if $\eta_{ik}x_k \geq \eta_{ij}x_j$ for all resources $i$, with at least one strict inequality.} 
In words, no other user's allocation would enable a user to process more jobs than her allocation would.
\end{definition}
\begin{proposition}
\cite[Proposition 5]{Joe13}
 For $\beta>0$ and $\lambda=(1-\beta)/\beta$, envy-freeness holds if $\beta>1$.
\end{proposition}
We then provide the proof of Proposition~\ref{prop:stochastic_envyfree}. 

\textbf{\textit{Proof of Proposition~\ref{prop:stochastic_envyfree}}.}
By Definition~\ref{stochastic_ParetoEfficiency}, if user $j$ envies user $k$'s share, $\eta_{ik}^\omega x_k \geq \eta_{ij}^\omega x_j$ for all resources $i$, with strict inequality for at least one resource. 
Now assume that user $j$ envies user $k$'s share. When $\lambda=(1-\beta)/\beta$ and $\beta>1$, the objective function of \eqref{dro_model_sampled} with the optimal solutions $\hat{z}$ and $\hat{p}^\omega$ is equivalent to optimizing the ($\alpha,p$)-fairness function by Corollary~\ref{corollary_convexity}. Similar to the proof of Proposition~\ref{stochastic_sharingincentive_prop}, we consider the first-order optimality conditions with Lagrange multipliers $\tau^\omega_i$:
\begin{equation}
   \frac{\partial L(x,\tau)}{\partial x}=\sum_{\omega=1}^{|\Omega|} \hat{p}^\omega\left(\mu_{j}^{\omega}\right)^{1-\beta} \left(x_{j}\right)^{-\beta}-\sum_{i=1}^{m} \sum_{\omega\in\hat{\Omega}(i)}\tau_{i}^\omega \eta_{i j}^\omega=0\; \forall j.\label{}  
\end{equation}
\noindent where, $\hat{\Omega}(i)=\{\omega|\hat{z}_{i}^\omega=1\}\; \forall i$, and for $\omega\not\in\hat{\Omega}(i)$, let $\tau_{i}^\omega=0$. We have:
\begin{equation}
\sum_{\omega=1}^{|\Omega|}\hat{p}^\omega\left(\mu_{j}^\omega x_{j}\right)^{1-\beta} =\sum_{i=1}^{m} \sum_{\omega\in\hat{\Omega}(i)}\tau_{i}^\omega \eta_{i j}^\omega x_{j}=\sum_{\omega=1}^{|\Omega|}\sum_{i=1}^{m} \tau_{i}^\omega \eta_{i j}^{\omega} x_{j}.
\end{equation}

\noindent We sum $\tau_{i}^\omega \eta_{i j}^\omega x_{j}$ and $\tau_{i}^\omega \eta_{i k}^\omega x_{k}$ for all resources $i$ and scenarios $\omega$.  As there is a strict inequality for one resource, we have: 
\begin{equation}
\begin{aligned}\label{envyfree_inequality_DRO}
\sum_{\omega=1}^{|\Omega|}\hat{p}^\omega\left(\mu_{j}^\omega x_{j}\right)^{1-\beta} &=\sum_{\omega=1}^{|\Omega|}\sum_{i=1}^{m} \tau_{i}^\omega \eta_{i j}^\omega x_{j}<\sum_{\omega=1}^{|\Omega|}\sum_{i=1}^{m} \tau_{i}^\omega \eta_{i k}^\omega x_{k} &=\sum_{\omega=1}^{|\Omega|}\hat{p}^\omega\left(\mu_{k}^\omega x_{k}\right)^{1-\beta}.\\
\end{aligned}
\end{equation}

\noindent If $ \left(\mu_{j}^\omega x_{j}\right)^{1-\beta} \geq \left(\mu_{k}^\omega x_{k}\right)^{1-\beta}$ for all $\omega$, then $$\sum_{\omega=1}^{|\Omega|}\hat{p}^\omega\left(\mu_{j}^\omega x_{j}\right)^{1-\beta}\geq \sum_{\omega=1}^{|\Omega|}\hat{p}^\omega\left(\mu_{k}^\omega x_{k}\right)^{1-\beta},$$ which violates \eqref{envyfree_inequality_DRO}. Therefore, to ensure the inequality is satisfied, there must exist an $\omega$ that 
$
\left(\mu_{j}^\omega x_{j}\right)^{1-\beta} <\left(\mu_{k}^\omega x_{k}\right)^{1-\beta}.$

\noindent As $\beta>1$, we have
$\mu_{j}^\omega x_{j}>\mu_{k}^\omega x_{k}$.
Since $\mu_{j}^\omega=\max\{\eta_{ij}^\omega\}$ and $\mu_{k}^\omega=\max\{\eta_{ik}^\omega\}$, if user $j$ envies user $k$'s share, then $\mu_{k}^\omega x_k\geq \mu_{j}^\omega x_j$, which is a contradiction.
Therefore, no user envies another user's allocation. 
$~\Box$

\section{$\epsilon$-Optimality}\label{optimality_proof}
\begin{proposition}\cite[Page 28]{Liu19}\label{beta_converge_dist}
Let $\xi^1, \xi^2,\dots,\xi^\omega$ be uniformly distributed on the support $\Xi\subseteq\mathbb{R}^d$. Then  $\beta^{|\Omega|}$ follows an extreme value distribution with 
$\lim_{|\Omega|\rightarrow \infty}\frac{|\Omega|(2\beta^{|\Omega|})^d-\log {|\Omega|}}{\log\log {|\Omega|}}=d-1$ with probability 1.
\end{proposition}

\begin{proposition}\cite[Proposition 3]{Lei24}\label{ambiguity_set_convergence}
For any $\mathbb{P}\in\mathcal{P}$, if  $\mathbb{P}^{|\Omega|}$ is a solution of $\text{min}_{\mathbb{Q}\in\mathcal{P}^{|\Omega|}}\rho(\mathbb{P},\mathbb{Q})$, then $\rho(\mathbb{P},\mathbb{P}^{|\Omega|})\leq C^H\beta^{|\Omega|}$, where $C^H$ is a constant value, $\beta^{|\Omega|}$ is defined in \eqref{beta_distance}, $\mathcal{P}$ is defined in \eqref{ambiguity_dro} and $\mathcal{P}^{|\Omega|}$ is defined in \eqref{ambiguity_dro_sampled}. 
\end{proposition}

\begin{proposition}\cite[Proposition 4]{Lei24}\label{convergence_chance_constraint} 
Let $\mathbb{P} \in \mathcal{P}$, $\mathbb{P}^{|\Omega|}$ is a solution of $\text{min}_{\mathbb{Q}\in\mathcal{P}^{|\Omega|}}\rho(\mathbb{P},\mathbb{Q})$.
Then $\left|\text{Prob}_\mathbb{P}\left(\sum_{j=1}^{d} \tilde{r}_{ij} \hat{x}_{j} \leq c_{i}\right)-\text{Prob}_{\mathbb{P}^{|\Omega|}}\left(\sum_{j=1}^{d} r_{ij}^\omega \hat{x}_{j} \leq c_{i}\right)\right|\leq \sqrt{y\beta^{|\Omega|}}$, where $y := 2\kappa^GC^\mathcal{P}C^H$ is a constant, $\beta^{|\Omega|}$ is defined in \eqref{beta_distance}, $\mathcal{P}$ is defined in \eqref{ambiguity_dro} and $\mathcal{P}^{|\Omega|}$ is defined in \eqref{ambiguity_dro_sampled}.
\end{proposition}

\begin{theorem}
If Assumption~\ref{feasible_assumption} is satisfied, then the inner problem \eqref{primal} has an optimal solution $\mathbb{P}_0^{|\Omega|}$ such that $q^{|\Omega|}(\hat{x})\leq q(\hat{x})+\epsilon$ when $|\Omega|$ is sufficiently large.
\end{theorem}

\begin{proof}

We first show that if $\mathbb{P}_0$ satisfies Assumption~\ref{feasible_assumption}, then $\exists \, \mathbb{P}_0^{|\Omega|}=$ $\text{argmin}_{\mathbb{P}^{|\Omega|}\in\mathcal{P}^{|\Omega|}} \rho(\mathbb{P}^{|\Omega|},\mathbb{P}_0)$ such that $\mathbb{P}_0^{|\Omega|}=\{p_0^1,\dots,p_0^{|\Omega|}\}, \underline{u}_{ij}-\epsilon\leq  \sum_{\omega=1}^{|\Omega|}p_0^\omega r_{ij}^\omega\leq \bar{u}_{ij}+\epsilon, \sum_{\omega=1}^{|\Omega|} p^\omega (r_{ij}^\omega)^2 -\left(\sum_{\omega=1}^{|\Omega|} p^\omega r_{ij}^\omega\right)^2\leq \bar{\sigma}_{ij}^2+\epsilon ,\forall i,j$, where $\rho$ is defined in \eqref{define_distanceP}.  According to Proposition~\ref{convergence_chance_constraint}, $\rho(\mathbb{P}^{|\Omega|}_0,\mathbb{P}_0)\leq C^H\beta^{|\Omega|}$, where $C^H$ is a constant and $\beta^{|\Omega|}$ is defined in \eqref{beta_distance}.
According to the definition of $\rho$ \eqref{define_distanceP}, 
$$
\begin{aligned}
    |\mathbb{E}_{\mathbb{P}_0}[\tilde{r}_{ij}]-\mathbb{E}_{\mathbb{P}_0^{|\Omega|}}[r^\omega_{ij}]|&\leq \sup_{g\in\mathcal{G}}|\mathbb{E}_{\mathbb{P}_0}[g(\xi)]-\mathbb{E}_{\mathbb{P}_0^{|\Omega|}}[g(\xi)]|\leq C^H\beta^{|\Omega|}, 
\end{aligned}
$$
where $\mathcal{G}':=\{g|g(\xi)=\xi,\forall \xi\in \Xi\}\subseteq \mathcal{G}$. Then if $\underline{u}_{ij}\leq \mathbb{E}_{\mathbb{P}_0}[\tilde{r}_{ij}]\leq \bar{u}_{ij}$, we have 
$\underline{u}_{ij}-C^H\beta^{|\Omega|}\leq \mathbb{E}_{\mathbb{P}_0^{|\Omega|}}[r^\omega_{ij}]\leq \bar{u}_{ij}+C^H\beta^{|\Omega|}.$ Therefore, when $|\Omega|$ is sufficiently large such that $\beta^{|\Omega|}< \epsilon/C^H$, we have $\underline{u}_{ij}-\epsilon\leq  \sum_{\omega=1}^{|\Omega|}p_0^\omega r_{ij}^\omega\leq \bar{u}_{ij} +\epsilon\quad\forall i,j$. 

Similarly, we have 
$$
\begin{aligned}
 & \big|\mathbb{E}_{\mathbb{P}_0}[(\tilde{r}_{ij}-\mathbb{E}_{\mathbb{P}_0}[\tilde{r}_{ij}])^2]-\mathbb{E}_{\mathbb{P}_0^{|\Omega|}}[(r_{ij}^\omega-\mathbb{E}_{\mathbb{P}_0}[r_{ij}^\omega])^2]\big|\leq 4M_\Xi\cdot\sup_{g\in\mathcal{G}}\big|\mathbb{E}_{\mathbb{P}_0}[g(\xi)]-\mathbb{E}_{\mathbb{P}_0^{|\Omega|}}[g(\xi)]\big|\leq 4M_\Xi C^H\beta^{|\Omega|},  
\end{aligned}
$$ 
where  $\mathcal{G}'':=\big\{g|g(\xi)=\frac{1}{4M_\Xi}(\xi-\mathbb{E}[\xi])^2,\forall \xi\in \Xi\big\}\subseteq \mathcal{G}$.  
Therefore, when $|\Omega|$ is sufficiently large such that $\beta^{|\Omega|}< \frac{\epsilon}{4M_\Xi C^H}$ (Proposition~\ref{beta_converge_dist}).

Finally, we show that $\mathbb{P}_0^{|\Omega|}$ can be an $\epsilon-$optimal solution. According to Theorem~\ref{coro:beta_converge_dist} in \ref{convergence_proof},  $F(\hat{x},\xi)$ is Lipschitz continuous on $\xi$ with constant $\kappa^F$. Therefore, we have 
$$
\begin{aligned}
 &  \big|q(\hat{x})-q^{|\Omega|}(\hat{x})\big|=\big|\mathbb{E}_{\mathbb{P}_0^{|\Omega|}}[F(\hat{x},\xi)]-\mathbb{E}_{\mathbb{P}_0}[F(\hat{x},\xi)]\big| \leq \kappa^F\sup_{g\in\mathcal{G}}\big|\mathbb{E}_{\mathbb{P}_0}[g(\xi)]-\mathbb{E}_{\mathbb{P}_0^{|\Omega|}}[g(\xi)]\big|\leq \kappa^F C^H\beta^{|\Omega|}. 
\end{aligned}
$$ Consequently, 
$q^{|\Omega|}(\hat{x})\leq q(\hat{x})+\kappa^F C^H\beta^{|\Omega|}$. Therefore, when $|\Omega|$ is sufficiently large such that
$\beta^{|\Omega|}<\min\Big\{\frac{\epsilon}{ C^H},\allowbreak \frac{\epsilon}{4M_\Xi C^H}, \frac{\epsilon}{\kappa^F C^H}\Big\}$, $\mathbb{P}_0^{|\Omega|}$ 
is an $\epsilon-$feasible solution and 
$q^{|\Omega|}(\hat{x})\leq q(\hat{x})+\epsilon$.
$~\Box$
\end{proof}

\section{Proof of the Cutting Surface Algorithm}\label{proof:algorithm}
\textbf{\textit{Proof of Theorem~\ref{thm:alg}.}}
Algorithm~\ref{alg:probcuts} is a modified version of \textit{Algorithm 1} by \cite{luo2019decomposition}. The proof of termination in finitely many iterations is similar to that used in the exchange method (\cite[\textit{Theorem 7.2}]{hettich1993semi}). 
Suppose the algorithm terminates at the end of the $\hat{t}$th iteration and $x_{\hat{t}}$ is an optimal solution of the problem $\max_{\lambda,x} \{\lambda\;\; s.t.\; g(x,\mathbb{P}^{|\Omega|})\geq 0, \; x\in\mathcal{X}'(\mathbb{P}^{|\Omega|}), \forall \mathbb{P}\in {\bar{\mathcal{P}}}_{\hat{t}}^{|\Omega|}\}$  . Since $\bar{\mathcal{P}}_{\hat{t}}^{|\Omega|}\subseteq \bar{\mathcal{P}}^{|\Omega|}$, then $\lambda_{\hat{t}} \geq \operatorname{Val}\eqref{joint}$. By the separation problem and termination criteria, $\min_{\mathbb{P}^{|\Omega|}\in\mathcal{P}^{|\Omega|}}g(x_{\hat{t}},\mathbb{P}^{|\Omega|})\geq g(x_{\hat{t}},\mathbb{P}^{|\Omega|}_{\hat{t}+1})-\epsilon/2\geq -\epsilon$.
Therefore, $x_{\hat{t}}$ is an $\epsilon$-optimal solution of the master problem \eqref{joint}.

We now show that Algorithm~\ref{alg:probcuts} returns a solution $\hat{x}$ which is an $\epsilon$-optimal solution of the proposed SA-DR model \eqref{dro_model_sampled}.  Suppose Algorithm~\ref{alg:probcuts} terminates at the end of the $\hat{t}$th iteration and returns a solution $(\hat{\lambda},\hat{x})$. Since the solution is $\epsilon$-feasible, for  $\hat{\lambda} \geq \hat{v}^{|\Omega|}$, it satisfies:
\begin{equation}
    \begin{aligned}
&(\hat{\lambda},\hat{x}) \in \text{argmax}_{\lambda,x}\Bigg\{ \lambda:s.t. -\lambda+\sum_{\omega=1}^{|\Omega|}p^\omega F(x,\xi^{\omega})\geq 0, \mathbb{P}^{|\Omega|}\in \bar{\mathcal{P}}_{\hat{t}}^{|\Omega|}, \\
&-\lambda+\sum_{\omega=1}^{|\Omega|}p^\omega F(x,\xi^\omega)\geq 0, \mathbb{P}^{|\Omega|}\in \bar{\mathcal{P}}^{|\Omega|}/\bar{\mathcal{P}}_{\hat{t}}^{|\Omega|},x\in\mathcal{X}'(\mathbb{P}^{|\Omega|})
 \Bigg\}.       
    \end{aligned}
\end{equation}
For any constant $\epsilon$, by redefining the decision variable $\lambda\rightarrow \lambda-\epsilon$, it follows that 
\begin{equation}
    \begin{aligned}
&q^{|\Omega|}(\hat{x})\\
& =\max_{\lambda}\Bigg\{ \lambda\;\Bigg|\,-\lambda+\sum_{\omega=1}^{|\Omega|}p^\omega F(\hat{x},\xi^\omega)\geq 0, \mathbb{P}^{|\Omega|}\in \bar{\mathcal{P}}^{|\Omega|}\Bigg\}\\
& = \max_{\lambda}\Bigg\{\lambda-\epsilon \;\Bigg|\, -\lambda+\epsilon+\sum_{\omega=1}^{|\Omega|}p^\omega F(\hat{x},\xi^\omega)\geq 0, \mathbb{P}^{|\Omega|}\in \bar{\mathcal{P}}^{|\Omega|}\Bigg\} \\
& = \max_{\lambda}\Bigg\{ \lambda\;\Bigg|\, -\lambda+\sum_{\omega=1}^{|\Omega|}p^\omega F(\hat{x},\xi^\omega)\geq -\epsilon, \mathbb{P}^{|\Omega|}\in \bar{\mathcal{P}}^{|\Omega|}\Bigg\}-\epsilon\\
& \geq  \max_{\lambda}\Bigg\{\lambda\;\Bigg|\, -\lambda+\sum_{\omega=1}^{|\Omega|}p^\omega F(\hat{x},\xi^{\omega})\geq 0,\mathbb{P}^{|\Omega|}\in \bar{\mathcal{P}}_{\hat{t}}^{|\Omega|}, \\
&\qquad \qquad  -\lambda+\sum_{\omega=1}^{|\Omega|}p^\omega F(\hat{x},\xi^\omega)\geq -\epsilon,\mathbb{P}^{|\Omega|}\in \bar{\mathcal{P}}^{|\Omega|}/\bar{\mathcal{P}}_{\hat{t}}^{|\Omega|}\Bigg\}-\epsilon\\
& = \hat{\lambda}-\epsilon\geq \hat{v}^{|\Omega|}-\epsilon.
\end{aligned}\nonumber
\end{equation}
The third inequality above is obtained from the relaxation of the constraints. 
The result follows from the definition of an $\epsilon$-optimal solution.
~$\Box$

\section{Closed-Form Specializations of $\mathbb{E}[F(x,\xi)]$ Across $\beta$}

Recall $y_j(\xi)=\tilde{\mu}_j x_j$, $S(\xi)=\sum_{k=1}^d y_k(\xi)$, and $\tilde t_j(x,\xi)=y_j(\xi)/S(\xi)$. From \eqref{DRO_obj},
\begin{equation}
F(x,\xi)=\operatorname{sign}(1-\beta)\Bigg[\sum_{j=1}^d\tilde t_j(x,\xi)^{\,1-\beta}\Bigg]^{\!1/\beta}\,S(\xi)^{\lambda}.
\label{X.1}
\end{equation}
We use boundedness of $\Xi$ and compactness of $\mathcal{X}$ (Section~\ref{subsec:stoch_model}) to justify exchanging limits with expectation via dominated convergence.

\paragraph{Case $\boldsymbol{\beta=1}$ (fairness ignored)}
Let $\phi(\beta):=\Big[\sum_j \tilde t_j^{\,1-\beta}\Big]^{1/\beta}$. Using continuity of $F(x,.)$, $\lim_{\beta\to 1}\phi(\beta)=1$. Hence,
\begin{equation}
\lim_{\beta\to 1}F(x,\xi)=S(\xi)^{\lambda}\quad\Rightarrow\quad \mathbb{E}[F(x,\xi)]=\mathbb{E}\big[S(\xi)^{\lambda}\big].
\label{X.2}
\end{equation}

\paragraph{Case $\boldsymbol{\beta\to 0}$ (entropy/Hill number)}
Write $\log \phi(\beta)=\tfrac{1}{\beta}\log\big(\sum_j \tilde t_j^{\,1-\beta}\big)$. Using $\,\tilde t_j^{\,1-\beta}= \tilde t_j\,e^{-\beta\log \tilde t_j}$ and a first-order expansion,
\[
\sum_j \tilde t_j^{\,1-\beta}=1-\beta\sum_j \tilde t_j\log\tilde t_j + o(\beta).
\]
Therefore $\lim_{\beta\to 0}\log\phi(\beta)= -\sum_j \tilde t_j\log\tilde t_j=:H(\tilde{\boldsymbol t})$, i.e.,
\begin{equation}
\lim_{\beta\to 0}\phi(\beta)=\exp\!\big(H(\tilde{\boldsymbol t})\big)\quad\Rightarrow\quad 
\mathbb{E}[F]=\mathbb{E}\!\Big[\exp\!\big(H(\tilde{\boldsymbol t})\big)\,S(\xi)^{\lambda}\Big].
\label{X.3}
\end{equation}

\paragraph{Case $\boldsymbol{\beta\in(0,1)}$ and $\boldsymbol{\beta\in(1,\infty)}$ ($\alpha$-fair “fairness component”)}
Using the identity from \cite{Lan2010} that for any share vector $t$, 
\[
\sum_j t_j^{\,1-\beta}=(1-\beta)\,U_{\beta}(p),\qquad 
U_{\beta}(t):=\sum_j \frac{t_j^{\,1-\beta}}{1-\beta},
\]
we can rewrite \eqref{X.1} as
\begin{equation}
F(x,\xi)=\operatorname{sign}(1-\beta)\,\big[(1-\beta)U_{\beta}(\tilde{\boldsymbol t})\big]^{1/\beta}\,S(\xi)^{\lambda}.
\label{X.4}
\end{equation}
For $\beta\in(0,1)$ this is positive; for $\beta>1$ it is negative (consistent with the deterministic map in Table~III, page~5 of \cite{Lan2010}). Taking expectation yields the entries in Table~\ref{tab:stoch_fairness_cases}.

\paragraph{Case $\boldsymbol{\beta=-1}$ (Jain)}
With $\beta=-1$, $\sum_j \tilde t_j^{\,1-\beta}=\sum_j \tilde t_j^{2}$ and $1/\beta=-1$, hence
\[
f_{-1}(\tilde{\boldsymbol y}):=\Big(\sum_j \tilde t_j^2\Big)^{-1}=\frac{S(\xi)^2}{\sum_j y_j(\xi)^2}.
\]
Plugging into \eqref{X.1} gives
\begin{equation}
F(x,\xi)=\frac{S(\xi)^{\lambda+2}}{\sum_j y_j(\xi)^2}=\frac{S(\xi)^{\lambda+2}}{\sum_j (\tilde{\mu}_j x_j)^2},
\label{X.5}
\end{equation}
and therefore $\mathbb{E}[F]$ equals the expectation of the right-hand side.

\paragraph{Case $\boldsymbol{\beta\to\infty}$ (max ratio) and $\boldsymbol{\beta\to-\infty}$ (min ratio)}
Using the pointwise limits of $f_{\beta}$ (Table~III, page~5, \cite{Lan2010}),
\[
\begin{aligned}
\lim_{\beta\to\infty} \phi(\beta)=-\max_j \frac{1}{\tilde t_j}=-\frac{S(\xi)}{\min_j y_j(\xi)},\\
\lim_{\beta\to-\infty} \phi(\beta)=\min_j \frac{1}{\tilde t_j}=\frac{S(\xi)}{\max_j y_j(\xi)}.    
\end{aligned}
\]
Multiplying by $S(\xi)^{\lambda}$ yields
\begin{equation}
\begin{aligned}
  \beta\to\infty:\; F(x,\xi)\to -\,S(\xi)^{\lambda+1}\max_{j}\frac{1}{y_j(\xi)},\\
\beta\to-\infty:\; F(x,\xi)\to \,S(\xi)^{\lambda+1}\min_{j}\frac{1}{y_j(\xi)},  
\end{aligned}
\label{X.6}
\end{equation}
and dominated convergence gives the expectations in Table~\ref{tab:stoch_fairness_cases}.

\paragraph{Remark on $\boldsymbol{\lambda=(1-\beta)/\beta}$ (the $\alpha$-fair scaling inside $F$)}
Consider the factor $S(\xi)$ from the inner sum in \eqref{X.1}:
\[
\begin{aligned}
 F(x,\xi) &=\operatorname{sign}(1-\beta)\Big[S(\xi)^{\beta-1}\sum_j y_j(\xi)^{\,1-\beta}\Big]^{1/\beta}S(\xi)^{\lambda}\\
&=\operatorname{sign}(1-\beta)\Big(\sum_j y_j(\xi)^{\,1-\beta}\Big)^{\!1/\beta}S(\xi)^{\lambda+\frac{\beta-1}{\beta}}.   
\end{aligned}
\]
Hence if $\lambda=\tfrac{1-\beta}{\beta}$ the scale factor cancels and
\begin{equation}
F(x,\xi)=\operatorname{sign}(1-\beta)\Big(\sum_j y_j(\xi)^{\,1-\beta}\Big)^{\!1/\beta},
\label{X.7}
\end{equation}
so $\mathbb{E}[F]=\mathbb{E}\!\left[\operatorname{sign}(1-\beta)\big(\sum_j (\tilde{\mu}_j x_j)^{\,1-\beta}\big)^{1/\beta}\right]$. 
This preserves the full structure of $F$ under expectation (we do not replace it by $\sum_j (\tilde{\mu}_j x_j)^{1-\beta}/(1-\beta)$, which is a different functional).
\hfill$\square$

\end{document}